\documentclass[a4paper,10pt]{article}

\usepackage[utf8]{inputenc}
\usepackage{amsmath,amssymb,amsthm} 
\newcommand{\R}{\mathbb{R}} 

\newcommand{\N}{\mathbb{N}}

\newcommand{\F}{\mathcal{F}} 
\newcommand{\FI}{\mathcal{F}^{-1}}
\renewcommand{\div}{\mathrm{div}} 
\newcommand{\supp}{\mathrm{supp}} 
 
\newcommand{\diam}{\mathrm{diam}} 
\newcommand{\meas}{\mathrm{meas}}

\newtheorem{theorem}{Theorem}
\newtheorem{proposition}{Proposition}
\newtheorem{corollary}{Corollary}
\newtheorem{lemma}{Lemma}
\newtheorem{definition}{Definition}
\newtheorem{remark}{Remark}
\newcommand{\no}{\text{\rm RTV}}
\newcommand{\BV}{\text{\rm TV}}
\newcommand{\TV}{\text{\rm TV}}
\newcommand{\argmin}{\mathrm{argmin}}
\newcommand{\Var}{\mathrm{Var}} 
 
\usepackage{todonotes}
 
\newcommand{\floor}[1]{\lfloor #1 \rfloor}
\newcommand{\norm}[1]{\left\|#1\right\|}
\newcommand{\abs}[1]{\left\vert#1\right\vert}

\newcommand{\kl}[1]{\left(#1\right)}
\newcommand{\Kl}[1]{\left\{#1\right\}}

\newcommand{\LpRn}{{L_p(\R^n)}}

\newcommand{\HspRn}{{H^{s,p}(\R^n)}}
\newcommand{\HspdRn}{{H^{s,p'}(\R^n)}}

\newcommand{\SsptRn}{{S^{s,p}_t(\R^n)}}
\newcommand{\SspdtRn}{{S^{s,p'}_t(\R^n)}}
\newcommand{\Rspt}{R^{s,p}_t}
\newcommand{\RsptSnR}{{R^{s,p}_t(\R\times S^{n-1})}}

\newcommand{\Rn}{{\R^n}}
\newcommand{\Snm}{{S^{n-1}}}
\newcommand{\Fs}{\mathcal{F}_1} 

\newcommand{\Dst}{\Psi_{s,t}}
\newcommand{\SsptRx}[1]{{S^{s,p}_t(\R^{#1})}}

\newcommand{\SsxtRn}[1]{{S^{s,#1}_t(\R^n)}}
\newcommand{\SsxyRn}[2]{{S^{s,#1}_{#2}(\R^n)}}
\newcommand{\ind}[1]{\chi_{#1}}
\newcommand{\prox}{{\rm prox}}
\newcommand{\scal}{c_{\operatorname{sc}}}
\newcommand{\obsolete}[1]{}
\newcommand{\ol}[1]{\overline{#1}}
\usepackage{hyperref}
\newtheorem{example}{Example}

\newcommand{\remove}[1]{}

\title{Norms in sinogram space 
and stability estimates for the Radon transform}
\author{Stefan Kindermann\footnote{%
Industrial Mathematics Institute, Johannes Kepler University Linz, 4040 Linz, Austria, 
({\tt kindermann@indmath.uni-linz.ac.at}).}
\and
Simon Hubmer\footnote{%
Industrial Mathematics Institute, Johannes Kepler University Linz, 4040 Linz, Austria,
({\tt simon.hubmer@jku.at}). Corresponding author.}
}

\begin{document}

\maketitle

\begin{abstract}
We consider different norms for the Radon transform $Rf$ of a function~$f$ and investigate under which conditions they can be estimated from above or below by some standard norms for~$f$. We define Fourier-based norms for $Rf$ which can be related to Bessel-potential space norms for~$f$. Furthermore, we define a variant of a total-variation norm for $Rf$ and provide conditions under which it is equivalent to the total-variation norm of~$f$. To illustrate potential applications of these results, we propose a novel nonlinear backprojection method for inverting the Radon transform and present numerical results on simulated and experimental data.


\end{abstract}

\section{Introduction}

In this article, we investigate norms for the Radon transform $Rf$ of a function $f$, in the following referred to as \emph{sinogram norms}, with the aim to provide lower and upper bounds for them with respect to some standard norms for~$f$. This is motivated by the objective of,  on the one hand, generalizing  the well-known Sobolev-space estimates and, on the other hand, constructing filter operators based on these sinogram norms which can be used in corresponding approximate inversion methods for the Radon transform. Since we aim to go beyond the Hilbert-space case, the obtained filters are in general nonlinear operators. 

First of all, recall that the $n$-dimensional Radon transform of a function $f: \R^n \to \R$ is defined by 
    \begin{equation*}
        Rf(\sigma,\theta) := \int_{x\cdot\theta = \sigma} 
        f(x) \, dx, \qquad (\sigma,\theta) \in Z:= \R \times S^{n-1} \,, 
    \end{equation*}
where the integration is over an $(n-1)$-dimensional hyperplane with offset $\sigma$ and normal vector $\theta \in S^{n-1}$. For background and further details, see, e.g.,~\cite{Natterer}.

Since inverting the Radon transform is well known to be an ill-posed problem, it requires regularization/filtering for its solution \cite{Louis_1989,Natterer}. Consider, for instance, Tikhonov regularization, where a least-squares functional is augmented with a penalty functional, which often takes the form of a norm $\|f\|$, i.e.,
    \begin{equation}\label{Radon_Tikhonov}
        \|Rf - y\|^2 + \alpha \|f\| \,.     
    \end{equation}
If, in this situation, we can find a norm $\|Rf\|_{S}$ which is equivalent to $\|f\|$, then we can instead consider the equivalent Tikhonov functional
    \begin{equation*}
        \|Rf - y\|^2 + \alpha \|Rf\|_{S} \,,   
    \end{equation*}
the minimizer of which is also expected to yield a reasonable solution. As we see in Section~\ref{sec:4}, such an approach allows for more flexible regularization/filtering approaches, since data noise removal and actual inversion can be separated. In the linear case, this has, for example, been used in the two-step regularization~\cite{Klann,Klann2} or the $Y$-scale regularization~\cite{Egger}.

Even though achieving norm equivalence of $\|f\|$ and $\|R f\|_S$ is perhaps a too ambitious goal, the idea can still be used in numerical computations as long as the sinogram norms have  reasonable filter properties. In any case, investigating the relations between norms in image space and sinogram space provides some new stability estimates for the Radon transform which might be of general use. 

In the Hilbert space case, equivalent norms have been found in the form of classical Sobolev-space estimates for the Radon transform, see, e.g.,~\cite{Natterer1980,Natterer}:
    \begin{equation}\label{sob}
        C_1 \|f\|_{H^s_0(\Omega^n) }\leq \|R f\|_{H^{s +\frac{n-1}{2}}(Z)}   \leq C_2 \|f\|_{H^s_0(\Omega^n)} \,, 
    \end{equation}
given that $f$ has compact support in the unit ball $\Omega^n$. Here, $H^s_0(\Omega^n)$ and $H^{s +\frac{n-1}{2}}(Z)$ are classical Sobolev spaces, whose norms can be defined via Fourier transforms as in Section~\ref{sec:2}. Note that the derivatives in $H^{s +\frac{n-1}{2}}(Z)$ are only taken with respect to the offset variable $\sigma$. Similar results for the $L^2$-case can, e.g., be found in \cite{Hertle,HoMi,RuQui} and the references in the next sections. Estimates in spaces of smooth functions with low decay can be found in \cite{kat}.

The objective of this article is to find useful norms for $Rf$ and certain conditions on $f$ such that \eqref{sob} (or parts of it) are satisfied when the Sobolev norms are replaced by other norms, such as the $W^{k,p}$-norms or total variation. An instance of a sinogram-space norm idea is to use a Besov-norm in Radon space, i.e., $\|Rf\|_{R_{p,q}^s}$, which leads to the Ridge space of Cand\`{e}s \cite{CandesPhD,CaDo}. However, relating this norm to function space norms of $f$ or considering them from the aspect of  variational regularization, as is our intent here, has not yet been attempted. Related, but different to our study, is the analysis of the mapping properties of the operator $R^*R$, which corresponds to the Riesz potential operator. Upper bounds in $L^p$ spaces such as \cite[Theorem~1]{Stein70} are based on Hardy-Littlewood-Sobolev estimates, while upper bounds in Besov spaces can be found, e.g., in \cite{Yang03}. A lower and upper bound for nonnegative functions in terms of maximal functions due to Muckenhoupt and Wheeden can be found in \cite[Theorem~3.6.1]{AdHe}.

This article is organized as follows: In Section~\ref{sec:2}, we extend \eqref{sob}, where the Sobolev norms are replaced by Bessel-potential space norms. Estimates which generalize \eqref{sob} are obtained using Sharafutdinov-type norms. This section is based on Fourier-space estimates. In Section~\ref{sec:3}, we investigate the use of the well-known total-variation norm for $f$ and a novel total-variation-like norm for~$Rf$. We obtain norm equivalence for certain simple images $f$, and the methods used there rely on results from geometric measure theory. In Section~\ref{sec:4}, we illustrate the use of these norms for deriving nonlinear filtered backprojection algorithms, and present a number of numerical results on simulated and experimental data.

\section{Bessel-potential space estimates}\label{sec:2}

We start by discussing generalizations of \eqref{sob} based on norms in Fourier space. Here, the aim is to replace the usual Sobolev spaces for $f$ by Bessel-potential spaces based on $L^p$, $p\not = 2$, and use appropriate Fourier-based norms for $Rf$. In the following, we denote by $p'$ the conjugate index of $p$, i.e., $\frac{1}{p} + \frac{1}{p'} = 1$.

\subsection{Differential dimension}

The classic stability estimate \eqref{sob} implies that in the standard Sobolev space setting, the Radon transform induces a smoothing in the offset variable of the order $\frac{n-1}{2}$. Consequently, inverting the Radon transform can be considered as being as ill-posed as differentiation of order $\frac{n-1}{2}$. This is in fact reflected in the classic back projection formula \cite{Natterer}, where just such a differentiation is performed. Hence, it is commonly thought that in order to obtain stability estimates for the Radon transform, the norm in sinogram space requires an $\frac{n-1}{2}$-times higher order of differentiation in the offset variable than that of the norm in $f$. While this paradigm is correct for $L^2$-based spaces, it is not necessarily true when generalizing \eqref{sob} to $L^p$-based spaces, which we indicate below.

Instead of the order of differentiation, the more appropriate concept in this case is that of the differential dimension (see, e.g., \cite{RuSi}). For a function $f$ defined in $\R^n$, let \mbox{$f_\lambda:= f(\lambda x)$}, for some $\lambda >0$. For a given norm, we define the differential dimension as the exponent $a \in \R$ such that $\|f_\lambda\| = \lambda^a \|f\|$, plus possibly lower-order terms in $\lambda$ (whenever such an exponent exists). It is well-known that the usual Besov/Sobolev spaces that use the $L^p$-norm of a highest derivative of order $s$ have a differential dimension (i.e., scaling exponent) given by 
    \begin{equation*}
        a= s -\frac{n}{p} \,.        
    \end{equation*} 
The Radon transform of $f_\lambda = f_\lambda(\sigma,\theta) := f(\lambda \sigma, \theta)$ satisfies 
    \begin{equation*}
        R f_\lambda(\sigma,\theta) = \frac{1}{\lambda^{n-1}} 
        R f(\lambda \sigma,\theta) \,.
    \end{equation*}
Thus, taking for $f$ an $L^p$-norm with order-$s$ derivatives (e.g., $\|f\|_{W^{s,p}}$) and for $Rf$ an $L^{p_R}$-norm with $s_R$ derivatives for the (one-dimensional) offset variable $\sigma$, (e.g., $\|Rf(\cdot,\theta)\|_{W^{s_R,p_R}}$), we have the necessary condition for norm-equivalence 
    \begin{equation}\label{noi} 
        s_{R} - \frac{1}{p_R} -(n-1)= s -\frac{n}{p} \,. 
    \end{equation}
In particular, this relation holds for \eqref{sob}, where $s_R = s + \frac{n-1}{2}$ and $p_R = p = 2$. Thus, rather than saying that the Radon transform increases smoothness by $\frac{n-1}{2}$ it is more suitable to say that the Radon transform increases the differential order by $n-1$. From these considerations, it follows that when the Radon transform is considered between $L^p$-based spaces, is does not necessarily induce a smoothing in the offset variable of order $\frac{n-1}{2}$ as in the classic $L^2$-case.

\subsection{Fourier-based norms and estimates}

In this section, we define some appropriate norms for the Radon transform and state a number of known norm estimates, starting with the classic Oberlin-Stein estimates \cite{SteinOberlin}: For $Rf = :g = :g(\sigma,\theta)$ in Radon space, define the norms:
    \begin{equation*}
        \|g\|_{q,r}^q = \int_{S^{n-1}} \left(\int_{\R} |g(\sigma,\theta)|^r \, d\sigma \right)^\frac{q}{r} \, d \theta \,.  
    \end{equation*}
Then due to \cite[Theorem 1]{SteinOberlin}, for $n\geq 2$ there is a constant $C>0$ such that 
    \begin{equation}\label{eq:SO}
        \|R f \|_{q,r}  \leq C \|f\|_{L^p} \,,
    \end{equation}
if and only if
    \begin{equation*}
        1 \leq p < \frac{n}{n-1} \,,
        \qquad q \leq p' \,, \qquad \text{and} \qquad \frac{1}{r} = \frac{n}{p} - n+1 \,. 
    \end{equation*}
Note that the third of these conditions, which essentially restricts the parameter $r$, corresponds to the requirement of scale invariance discussed above.

Next, we discuss Reshetnyak-Sharafutdinov estimates. First, we need to introduce some norms based on the Fourier transform, which in the following is denoted by $\F$ and defined via $\F f(\xi) := (2\pi)^{-n/2} \int_{\R^n} e^{i x \cdot\xi} f(x) dx$. 

\begin{definition}
For $s,t \in \R$ and $\xi \in \R^n$, define the weights
    \begin{align}  
        w_s(\xi) := (1 + |\xi|^2)^\frac{s}{2} \,, \qquad \text{and} \qquad
        v_{s,t}(\xi) :=  (1 + |\xi|^2)^\frac{s}{2}
        \left(\frac{|\xi|}{(1 + |\xi|^2)^\frac{1}{2} } \right)^t \,.
        \label{vweight}
    \end{align}
Furthermore, define the standard Bessel-potential spaces $\HspRn$ via
    \begin{equation}\label{defbessel}
        \norm{f}_\HspRn :=  \| \F^{-1}\left(w_{s}\F(f) \right) \|_{L^p} \,.
    \end{equation}
For all $1\leq p < \infty$, $s\in \R$, and $t > -n/p$, the spaces $\SsptRn$ are defined as the completion of $\mathcal{S}(\Rn)$ under the norm
    \begin{equation} 
        \norm{f}_\SsptRn  := \| \F(f) v_{s,t} \|_{L^p} \,. 
    \end{equation}
Finally, for all $1\leq p < \infty$, $s\in \R$, and $t > -1/p$, define the sinogram space $\RsptSnR$ as the completion of $\mathcal{S}(\R \times S^{n-1})$ with respect to the norm
    \begin{equation*}
    \begin{split}
        \norm{g}_\RsptSnR &:= \left( \int_{S^{n-1}}  \| g(\cdot,\theta)\|_{\SsptRx{1}}^p \, d \theta \right)^{1/p}  \\
        &= \left(\int_{S^{n-1}} \int_{\R}
         |v_{s,t}(\sigma)|^p |\Fs g(\sigma,\theta)|^p d\sigma d\theta \right)^{1/p}\,.
     \end{split}
\end{equation*}
\end{definition}

In the above definition, $\| g(\cdot,\theta)\|_{\SsptRx{1}}$ denotes the norm with respect to the offset variable $\sigma$ and thus only involves the 1D Fourier transform $\Fs g$.

For an integer $s \in \N_0$ and $1 < p < \infty$, the Bessel potential norm $\HspRn$ is equivalent to the usual Sobolev norm $W^{s,p}$, and for noninteger $s$, the space ``almost'' coincides with $W^{s,p}$; see, e.g., \cite[p.~12]{Triebel} or \cite[Theorem~7.63]{Adams}. Furthermore, it is known that the space $H^{s,p}$ coincides with the Lizorkin-Triebel space $F_{p,2}^s$; see \cite[Theorem~4.2.2]{AdHe}, \cite[p.~29]{Triebel}. The norms in $\RsptSnR$ can be seen as having differentiation order $s$, where $t$ is a fine-tuning parameter. With respect to the differential dimension, note that $\F f_\lambda = \lambda^{-n} \F f(\cdot/\lambda)$, and thus $\Rspt$ has a differential dimension of $ s -\frac{n}{p'}$. 

For the $L^2$-case, Sharafutdinov in \cite[Theorem 2.1]{Shara} generalized a result which, according to \cite{Shara}, originates in unpublished work of Reshetnyak (cf.~also \cite{Gelfand} and the references in \cite{Shara}): For all $s \in \R$ and $t > -\frac{n}{2}$ there holds  
    \begin{equation}\label{sharaeq}
        \norm{f}_\SsxtRn{2} = 
        \norm{Rf}_{R^{s + (n-1)/2,2}_{t+(n-1)/2}(\R \times \Snm)}\,, 
    \end{equation}
and the Radon transform is a bijective isometry between these spaces. Moreover, Sharafutdinov showed that $\norm{f}_\SsxtRn{2}$ and the usual Sobolev norm $\|f\|_s = \norm{f}_{H^{s,2}(\Rn)}$ are equivalent for compactly supported functions; see \cite[Lemma~2.3]{Shara}:

\begin{theorem}[\mbox{\cite[Lemma~2.3]{Shara}}]\label{thm_SHA_1}
For any $s\in\R$ and for $t \in (-n/2,n/2)$, the norms $\norm{f}_{S^{s,2}_{t}(\Rn)}$ and $\norm{f}_{S^{s,2}_0(\Rn)} = \norm{f}_{H^{s,2}(\Rn)}$ are equivalent on the space of functions $f \in \mathcal{S}(\Rn)$ supported in a fixed ball $\Kl{x \in \Rn \, \vert \, \abs{x} \leq A}$.
\end{theorem}

Hence, under the assumptions of Theorem~\ref{thm_SHA_1} and using \eqref{sharaeq}, there holds
\begin{equation*}
    C_1 \norm{f}_{H^{s,2}(\Rn)} \leq  \norm{Rf}_{R^{s + (n-1)/2,2}_{t+(n-1)/2}(\R \times \Snm)}
    \leq C_2 \norm{f}_{H^{s,2}(\Rn)} \,,
\end{equation*}
which yields a norm equivalence extending the classic Sobolev estimate~\eqref{sob}.

\subsubsection{Generalization}

In this section, we extend the Sharafutdinov result \eqref{sharaeq} to the $L^p$-case: 

\begin{theorem}
For all $1\leq p < \infty$, $s\in \R$, and $t > -n/p$, there holds 
    \begin{equation*}
        \norm{f}_\SsptRn = \norm{Rf}_{R^{s + (n-1)/p,p}_{t+(n-1)/p}(\R \times \Snm)} \,,
        \qquad
        \forall \, f \in \SsptRn \,.
    \end{equation*}
\end{theorem}
\begin{proof}
Let $\Fs$ denote the one-dimensional Fourier transform with respect to the offset variable $\sigma$ in sinogram space. Following \cite{Shara} and using the Fourier slice theorem \cite[Theorem 1.1]{Natterer}, i.e., $\Fs({Rf})(\sigma,\theta) = \F{f}(\sigma \theta)$\cite{Natterer}, we find 
    \begin{equation*}
    \begin{split}
        &\norm{Rf}_{R^{s + (n-1)/p,p}_{t+(n-1)/p}(\R \times \Snm)} ^p \\
        & \qquad =
        \frac{1}{2} \int_\Snm \int_\R  \left|v_{s + (n-1)/p,t+(n-1)/p}(\sigma)\right|^p \abs{\Fs(Rf)(\sigma,\theta) }^p \, d \sigma \, d\theta\\
        &\qquad =  
        \int_\Snm \int_0^\infty (1 + \sigma^2)^{p(s-t)/2} \abs{\sigma}^{pt+n-1} \abs{\Fs(Rf)(\sigma,\theta) }^p \, d \sigma \, d\theta
        \\
        &\qquad= 
        \int_\Rn (1 + \abs{\xi}^2)^{p(s-t)/2} \abs{\xi}^{pt} \abs{\Fs(Rf)(\xi/\abs{\xi},\abs{\xi}) }^p \, d \xi 
        \\
        &\qquad= 
        \int_\Rn (1 + \abs{\xi}^2)^{p(s-t)/2} \abs{\xi}^{pt} \abs{\F(f)(\xi) }^p \, 
        d \xi = \norm{ f }_\SsptRn^p \,,
    \end{split}
    \end{equation*}
which yields the assertion.
\end{proof}

The above theorem extends \eqref{sob} with the norm in $\SsptRn$ used for the images $f$. The next goal is to relate this norm to more commonly used penalty norms; our target are the Bessel-potential norms $\norm{f}_\HspRn$. Thus, we would like to  find embeddings between $\SsptRn$ and the Bessel-potential spaces. 

\begin{proposition}
For all $1 < p \leq 2$ and $t \geq 0$, there is a constant $C>0$ independent of $f$ such that 
    \begin{equation}\label{eq_helper_01}
        \norm{f}_{\SspdtRn} \leq  C \norm{f}_{\HspRn} \,.
    \end{equation}
For $s \geq  n(p'-2)/p'$ there exists a constant $c >0$ independent of $f$ such that 
    \begin{equation} \label{eq_helper_02}
        c\norm{f}_{H^{s-n(p'-2)/p',p'}(\Rn)} \leq 
        \norm{f}_{\SsxyRn{p'}{n(p'-2)/p'} } \,.
    \end{equation}
\end{proposition}
\begin{proof}
First, note that by their definition \eqref{vweight} it follows that $v_{s,t} \leq w_s$ for $t \geq 0$. Hence, using the Hausdorff-Young inequality we obtain
    \begin{align*}
        \norm{f}_{\SspdtRn} = \| v_{s,t} \F(f) \|_{L^{p'}}
        &\leq \| w_s \F(f)\|_{L^{p'}}
        = \|\F \F^{-1}( w_s \F(f))\|_{L^{p'}} \\
        & \leq    \|\F^{-1}( w_s \F(f))\|_{L^p}     
        = \norm{f}_{\HspRn} \,,
    \end{align*}
which yields \eqref{eq_helper_01} with $C=1$. To show the lower bound \eqref{eq_helper_02}, let $ f\in \mathcal{S}(\Rn)$ with $\norm{f}_{\SspdtRn} < \infty$. Setting $t =  n(p'-2)/p'\geq 0$ and $s \geq t$, we have
    \begin{align}
        \infty >\norm{f}_{\SspdtRn} 
        &=
        \kl{ \int_\Rn (1+\abs{\xi}^2)^{p'(s-t)/2} \abs{\xi}^{p't} \abs{\F(f)(\xi)}^{p'} \, d \xi}^{1/p'} \nonumber 
        \\
        & =
        \kl{ \int_\Rn  \abs{\xi}^{n(p'-2)} (1+\abs{\xi}^2)^{p'(s-t)/2} 
        \abs{\F(f)(\xi)}^{p'} \, d \xi}^{1/p'}\,, \label{thiseq}
    \end{align}
and since $s-t\geq 0$, it follows that $(1+\abs{\xi}^2)^{(s-t)/2} \F(f)(\xi) \in \mathcal{S}(\Rn)$ and by \eqref{thiseq}, 
    \begin{equation}\label{eq:wurst} 
        (1+\abs{\xi}^2)^{(s-t)/2} \F(f)(\xi) \frac{1}{|\xi|^{n(1-\frac{2}{p})}}
        = 
        (1+\abs{\xi}^2)^{(s-t)/2} \F(f)(\xi)  \abs{\xi}^{n(1-\frac{2}{p'})}  \in L^{p'}. 
        \end{equation}
The following inequality of Paley~\cite[Theorem~1.4.1]{Bergh} states that there is a constant $C_p>0$ such that for all $h \in L^{p}$,
    \begin{equation}\label{eq:Payley}
        \kl{\int_{\Rn} \abs{\xi}^{n(p-2)}\abs{\F(h)(\xi)}^p \,d\xi  }^{1/p} \leq C_{P} \norm{h}_\LpRn \,,
        \qquad
        1 < p \leq 2 \,.
    \end{equation}
Now let
    \begin{equation}\label{helper_g}
        g(\xi) := (1+\abs{\xi}^2)^{(s-t)/2} \F(f)(\xi)  \abs{\xi}^{n(1-\frac{2}{p'})} \,.
    \end{equation} 
Then, by \eqref{eq:wurst}, $g \in L^{p'}$ and $g(\xi) |\xi|^{n(1-\frac{2}{p})} \in \mathcal{S}(\Rn)$. For arbitrary $h \in  \mathcal{S}(\Rn)$, using \eqref {eq:Payley} together with the H\"older inequality we thus find that 
    \begin{align*}
        C_p \|h\|_{L^p} \|g\|_{L^{p'}} 
        &\geq 
        \int_{\Rn} |\xi|^{n(1 -\frac{2}{p})} \F(h)(\xi) \ol{g}(\xi) d\xi  =    \int_{\Rn}   h(\xi) \F  \left[  \ol{|\xi|^{n(1 -\frac{2}{p})} g(\xi)} \right]d\xi 
        \\
        &=
        \int_{\Rn}   h(\xi)  \ol{ \F^{-1} \left[|\xi|^{n(1 -\frac{2}{p})} g(\xi)\right] }d\xi 
        \\
        &=\int_{\Rn} h(\xi)  \ol{ \F^{-1} \left[ (1+\abs{\xi}^2)^{(s-t)/2} \F(f)(\xi) \right] } d \xi \,,
    \end{align*}
noting that the integral exists as the integrand is in  $\mathcal{S}(\Rn)$. Taking the supremum over $h \in L^p$ with $\|h\|_{L^p} = 1$ we get an $L^{p'}$ norm estimate, which yields
    \begin{align*}  
        \norm{f}_{H^{s-n(p'-2)/p',p'}(\Rn)}
        &=
        \left\|\F^{-1} \left[ (1+\abs{\xi}^2)^{(s-t)/2} \F(f)(\xi)\right]\right\|_{L^{p'}}
        \\
        &
        \leq
        C_p \|g\|_{L^{p'}}
        \overset{\eqref{helper_g}}{=}
        C_p  \norm{f}_{\SspdtRn} \,,
    \end{align*}
for $f \in \mathcal{S}(\Rn) \cap \SspdtRn$. By a limiting process, we thus obtain the lower bound~\eqref{eq_helper_02} for all $f \in \SspdtRn$ with $c = C_p^{-1}$, which completes the proof.
\end{proof}

Summing up, we arrive at the main theorem in this section: 

\begin{theorem}\label{mainshaf}
Let $1 < p \leq 2$ and $s \geq  n(p'-2)/p'$. Then there exists constants $C,c > 0$ independent of $f$ such that
    \begin{equation}\label{main1}
        c\norm{f}_{H^{s-n(p'-2)/p',p'}(\Rn)} \leq \norm{Rf}_{R^{s + (n-1)/p',p'}_{(np'-n- 1)/p'}(\R \times \Snm)} \leq C \norm{f}_{\HspRn} \,.
    \end{equation}
\end{theorem}

Informally speaking, the norm in \eqref{main1} for $Rf$ requires a differentiation (wrt.~$\sigma$) of order $s_R :=s + (n-1)/p'$, while the lower bound involves less differentiation of $f$ of order $s-n(p'-2)/p' = s_R -(n (p'-1) -1)/p'$. We may interpret this as an indication that the ill-posedness of inverting the Radon transform is similar to that of an $(n (p'-1) -1)/p'$ -fold differentiation. On the other hand, considering the upper bound in \eqref{main1}, we observe an increase of $(n-1)/p'$ in the order of differentiation of $Rf$ compared to that in the norm of $f$, indicating a smoothing property of $R$ of order $(n-1)/p'$, which, in case $p\not = 2$, is different from the observed ill-posedness in the lower bound.

However, the picture is more consistent when considering the differential dimension: It can be observed that the differential dimension of the lower and upper bounds agree (both are  $s - n/p$). Furthermore, the differential dimension of the norm for $Rf$ is $(s - n/p) + n-1$, and thus both upper and lower bounds are in agreement with \eqref{noi}. The observed discrepancy in the differential order for the lower and upper bounds might be due to the different $L^p$-spaces in the estimate. Replacing the norm in the lower bound by the same norm as in the upper one appears to be quite difficult.

\begin{remark}
By defining, for $s,t \in \R$, the pseudo-differential operator 
    \begin{equation*}
        \Dst(f)(\xi) := \FI\kl{(1+\abs{\xi}^2)^{s/2} \abs{\xi}^{t} \F(f)(\xi) } \,, 
    \end{equation*}
the previous theorem can be slightly generalized: With similar estimates as in the proofs of the above theorems, it follows that for all $1 < p \leq 2$ and all $0 \leq t \leq s$,
    \begin{equation*}
        c\norm{\Psi_{-t,t-n(p'-2)/p'}f}_\HspdRn \leq \norm{Rf}_{R^{s + (n-1)/p',p'}_{t+(n-1)/p'}(\R \times \Snm)} \leq C \norm{f}_{\HspRn} \,.
    \end{equation*}    
Again, the operator $\Psi$ at the lower bound can be seen as a generalized differentiation operator of order $n(2-p')/p'$, similar to \eqref{main1}.
\end{remark}

Since in applications $f$ represents an image, it is of high interest to use appropriate function spaces which model natural images. One of the most popular ones is the space of total (or bounded) variation. The next section is devoted to the study of conditions under which one can establish estimates of the form
    \begin{equation}\label{stability_TV}
        c \abs{f}_{TV} \leq \norm{ R f }_S \leq C \abs{f}_{TV} \,, 
    \end{equation}
which are extensions of \eqref{sob} where the TV semi-norm $\abs{\cdot}_{TV}$ defined below is used in the image space, and the still to be specified norm $\norm{ \cdot }_S$ in the sinogram space.

\section{Total variation in image and sinogram space}\label{sec:3}

In this section, we introduce a new norm in sinogram space which should emulate the behaviour of, and if possible be equivalent to, the total variation norm of $f$, in particular allowing for jumps in images. First, recall the definition of the $\TV$-seminorm on $\R^n$
    \begin{equation*}
        |f|_{TV} := \sup_{\phi\in [C_0^\infty(\R^n)]^n\,, \vert \phi(\cdot)\vert \leq 1}
        \int_{\R^n} f(x) \div \phi(x) \, dx \,.
    \end{equation*} 
For smooth functions $f$, this agrees with the Sobolev (semi)norm in $W^{1,1}$, and for $f$ being the indicator function this yields the perimeter functional, which for a (sufficiently regular) set gives the length of the boundary; see, e.g., \cite{Ambrosio}.

The main result below is that in the two-dimensional case $n=2$, and with the sinogram norm $\norm{\cdot}_{\no}$ defined below, the stability estimate \eqref{stability_TV} holds for all $f$ which are indicator functions of convex sets. 

In the following, we always work under the standard assumptions that $f$ has compact support (without loss of generality) in the unit ball, i.e., 
    \begin{equation}\label{supp} 
        \supp{f} \subset \{ x \in \R^n \,|\, |x| < 1 \} = : B_1 \,,  
    \end{equation}
which in fact turns $|\cdot|_{\TV}$ into a norm. For deriving certain lower bounds, it is sometimes necessary to additionally impose the nonnegativity constraint $f \geq 0 $. We denote by $\TV$ the space of functions with $|f|_{\TV}<\infty$ satisfying \eqref{supp}. Our candidate for a norm in sinogram space emulating $|f|_{\TV}$ is the following:

\begin{definition}
The norm $\| \cdot \|_{\no}$ in sinogram space is defined by 
    \begin{equation}\label{eq:nodef}
        \| g\|_{\no}  := \sup_{\theta \in S^{n-1}} | g(\cdot,\theta)|_{\TV} \,.
    \end{equation}
\end{definition}

In the above definition, $\| g\|_{\no}$ is the (one-dimensional) total variation in the offset variable $\sigma$ with the maximum taken over all angles. We denote by $\no$ the space of functions 
$g(\sigma,\theta)$ satisfying 
$\supp{\,g(\cdot,\theta)} \in [-1,1]$ (i.e., the analog of 
\eqref{supp}) and having  finite $\|\cdot\|_{\no}$-norm. 

Since for sufficiently smooth functions there holds $\TV = W^{1,1}$, it follows that for sufficiently smooth $g$, the norm $\| g\|_{\no} $ can be written as
    \begin{equation}\label{two} 
        \| g\|_{\no}  = \sup_{\theta \in S^{n-1}} \int_{\R} \left|\frac{\partial}{\partial \sigma} g(\sigma,\theta) \right| \, d\sigma \,, 
    \end{equation}
but, of course, the definition~\eqref{eq:nodef} allows also for nondifferentiable functions. 

Let us illustrate why we think that \eqref{eq:nodef} is a reasonable norm for the Radon transform  $g = R f$ being close to $|f|_{TV}$. First, the $L^1$-norm of the $\sigma$-derivative is motivated by the $L^1$-paradigm of compressed sensing (and the $\TV$-norm itself), which leads to this preference over other norms like $L^p$-norms of the derivative. 

The second motivation comes from the Oberlin-Stein estimates \eqref{eq:SO}. Informally, we have $R\nabla f \sim  \frac{d}{d\sigma} Rf $, and taking $p =1$, $r = 1$, and $q = \infty$ yields the upper bound $\|Rf\|_{\no} \leq C \|\nabla f\|_{L^1}$, which is the $\TV$-norm for smooth functions. We prove this estimate below and extend it to the $\TV$-space.

Considering the differential dimension, we observe that the $\TV$-norm in $n$ dimensions, with $s=1,p=1$, has differential dimension $ s - \frac{n}{p} = 1-n$, while the $\no$-norm has $s = 1,p=1,n=1$ and thus differential dimension $s - \frac{1}{1} = 0$. In particular, the differential dimension is increased by $n-1$ in sinogram space as required by \eqref{noi}. Thus, the $\no$-norm  satisfies the same scaling identities under $f \to f_\lambda$ as the $\TV$-norm.

Furthermore, it is not difficult to verify that $\|Rf\|_{\no}$ is invariant under shifts of $f$, which is not the case when using a $\theta$-derivative in place of $\frac{d}{d\sigma}$, and the space $\no$ allows for jumps in $f$, thus is not oversmoothing. 

The scaling property with respect to the differential dimension and the shift invariance is also satisfied by norms that use some $L^q$-norm, $q < \infty$, with respect to the $\theta$-variable: 
    \begin{equation*}
        \| g\|_{R,p}^p := \int_{\theta \in S^{n-1}}  | g(\cdot,\theta)|_{\TV}^p \, d \theta \,.
    \end{equation*}
However, as we shall see in one of the examples below, such a choice does not even yield an equivalence with the $\TV$-norm in simple cases.

\begin{example} \rm
As an example, let us consider $f =\ind{B_r}$, the indicator function of a ball in $\R^2$ with radius $r$. Its Radon transform can be easily calculated as 
    \begin{equation}\label{radsqu} 
        R f(\sigma,\theta) = \begin{cases} 2 \sqrt{r^2-\sigma^2} \,,  &\sigma \in [-r,r] \,,  \\ 0 \,, & \text{ else } \,. \end{cases} 
    \end{equation}
We have $\frac{\partial}{\partial \sigma} Rf(\sigma,\theta) = -2 \frac{\sigma}{\sqrt{r^2-\sigma^2}} $ for $\sigma \in (-r,r)$. The $\TV$-norm with respect to $\sigma$ is obtained by integrating $\left|\frac{\partial}{\partial \sigma} R(\sigma,\theta)\right|$ over $(-r,r)$, which gives 
    \begin{equation*}
        \|R f\|_{\no} = |R f(\cdot,\theta)|_{\TV} = 4 r \,.
    \end{equation*}
The total variation of $f$ is the perimeter $2 \pi r $, hence both norms are comparable for all radii $r$. 
\end{example}

\begin{example} \rm
In a second example, we consider $f = \ind{\Omega}$ the indicator function of a rectangle $\Omega = [-a,a] \times [-1,1]$. The Radon transform of $f$ can be computed, but involves a lengthy formula. It can be shown that $\sigma \to R(\sigma,\theta)$ is a continuous piecewise linear function, except when $\theta$ lies in the same direction as the rectangle's side, in which case it is the indicator function of an interval. As we show below, for such functions, the total variation is equal to $2 \max_{\sigma} R(\sigma,\theta)$, i.e., twice the length of the largest chord in direction $\theta^\bot$. The largest chord is found by taking a line through a corner point into direction $\theta^\bot$. Let $\theta^\bot = (\cos(\alpha),\sin(\alpha))$. By elementary geometry, we have 
    \begin{equation*}
        \max_{\sigma} R(\sigma,\theta) = \begin{cases}
        \frac{2 a}{\cos(\alpha)}\,, & \alpha \in [0,\arctan(a^{-1})] \,, \\ 
        \frac{2}{\sin(\alpha)} \,, &\alpha \in [\arctan(a^{-1}) \,, \frac{\pi}{2}] \,, 
        \end{cases}
    \end{equation*}
and the function can be symmetrically extended to the remaining angles $\alpha$. The maximum over $\theta$ (or $\alpha$) is attained at $\arctan(a^{-1})$, which is the angle of the diagonal, and its length is $2\sqrt{1 + a^2}$. Thus, 
    \begin{equation*}
        \|R f\|_{\no} = 4  \sqrt{1 + a^2} \,,
    \end{equation*}
which is comparable to $|f|_{TV}$ for all $a$ since this is given by the perimeter $2(1+a)$.

Note that the situation is different when we take the $L^1$-norm with respect to $\theta$, i.e., $\int |Rf(\cdot,\theta)|_{\TV} d\theta$. It can be calculated that
    \begin{equation*}
        \lim_{a\to 0} \int_{S^1} |R f(\cdot,\theta)|_{\TV} d \theta  = 0 \,,
    \end{equation*}
while the perimeter $|f|_{\TV}$ tends to $2$ as the side $a$ vanishes. Thus, this norm is {\em not} comparable to $|f|_{\TV}$, and this is yet another argument for taking the supremum in $\theta$ instead of an $L^p$-norm (for which a similar non-equivalence result can be shown in this example). 
\end{example}

\begin{remark} \rm 
Note that recently, a variant of the $\no$-norm (with the $L^1$-norm in the angle variable) has been used in the approximation theory related to learning  by shallow neural networks~\cite{Shuai,Ongie}.
\end{remark}

\subsection{Estimates and Properties}

For our TV estimates, we first need to show some properties of the $\no$-norm.  

\begin{proposition}
Let $f_n \to f$ in $L^1$. Then $f\to \|R f\|_{\no}$ is lower semicontinuous:
    \begin{equation}
        \|R f\|_{\no} \leq \liminf_{n\to \infty} \|R f_n\|_{\no} \,.
    \end{equation}
\end{proposition}
\begin{proof}
From the representation of the one-dimensional \TV-norm, we have 
    \begin{equation}\label{xyp}
        \|R f\|_{\no} = \sup_{\theta, \phi_{\theta}(\cdot) 
        \in C_0(\R) \,, |\phi_\theta(\cdot)| \leq 1} 
        \int_{\R} Rf(\sigma,\theta) \phi_\theta'(\sigma) \, d\sigma \,. 
    \end{equation}
By continuity of $R:L^1 \to L^1$, we may conclude that $Rf_n(\cdot,\theta) \to Rf(\cdot,\theta)$ in $L^1(\R)$ and consequently for fixed $\theta$ and $\phi_\theta$ there holds
    \begin{equation*}
        \int_{\R} Rf_n(\sigma,\theta) \phi_\theta'(\sigma) \, d\sigma \to 
        \int_{\R} Rf(\sigma,\theta) \phi_\theta'(\sigma) \, d\sigma \,.
    \end{equation*}
The lower semicontinuity then follows in a standard way; cf.~\cite[Remark 3.5]{Ambrosio}.
\end{proof}

Recall that a mollifier \cite[p.~41]{Ambrosio} $\rho$ in dimension $n$ is a $C^\infty(\R^{n})$-function with support in the unit ball which satisfies nonnegativity $\rho\geq 0$, symmetry \mbox{$\rho(-x) = \rho(x)$}, and the normalization $\int \rho(x) dx = 1$. The corresponding scaled mollifier for $\epsilon >0$, $\rho_\epsilon = \epsilon^{-n} \rho(\cdot/\epsilon)$ can be used for smoothing: The approximation $f_\epsilon :=f*\rho_\epsilon$ is smooth, and for $f \in \BV(\R^{n})$ we have (see \cite[Proposition~3.7]{Ambrosio})
    \begin{equation}\label{folula}
        |f|_{\TV} = \lim_{\epsilon \downarrow 0}
        \int_{\R^{n}} |\nabla f_\epsilon| \, dx\,.  
    \end{equation}
As a consequence, it is enough to verify estimates for the $\|\cdot\|_{\no}$-norm versus the $\TV$-norm for smooth functions: 

\begin{lemma}\label{lemma2}
Assume that either of the estimates 
    \begin{equation}\label{tt1} 
        \|Rf\|_{\no} \leq C_1 |f|_{\TV}
        \qquad \text{ or } \qquad 
        \|Rf\|_{\no} \geq C_2 |f|_{\TV} 
    \end{equation}
holds with some constants $C_1$ or $C_2$, for all smooth functions $f$. Then the respective estimates also hold for $f \in \TV$ with the same constants. 
\end{lemma}
\begin{proof}
Let $\theta$ be fixed, and take a tensor product mollifier aligned with the hypersurface $x\cdot\theta = \sigma$. Let this hypersurface be parameterized by 
    \begin{equation}\label{para}
        \{ x\,|\, x\cdot\theta = \sigma\} = 
        \{ \sigma \theta + t \,|\, t \in \R^{n-1} \} \,,
    \end{equation}
where $t$ are parameterizations of the hypersurface. Define $\rho(x) := \rho_1(\sigma) \rho_{n-1}(t)$ for $x = \sigma\theta + t \theta^\bot$, 
where $\rho_1$ is a 1D-mollifier and $\rho_{n-1}$ an $n-1$-dimensional mollifier. The scaled mollifier satisfies $\rho_{\epsilon}(x) = \epsilon^{-n} \rho(\frac{x}{\epsilon}) = \rho_{1,\epsilon}(\sigma) \rho_{n-1,\epsilon}(t)$. It follows from the convolution formula for the Radon transform \cite[Theorem~1.2]{Natterer} that
    \begin{equation*}
        R (f*\rho_\epsilon)(\sigma,\theta)  =
        Rf(\cdot,\theta) *_\sigma R(\rho_\epsilon)(\cdot,\theta) = 
        Rf(\cdot,\theta) *_\sigma \rho_{1,\epsilon}(\cdot) \,,
    \end{equation*}
where we have used the alignment of the mollifiers with the hypersurface and that the integrals of $\rho_{n-1}$ are $1$, such that $R(\rho_\epsilon)(\sigma,\theta) = \rho_{1,\epsilon}(\sigma)$. Since $\rho_{1,\epsilon}$ is a 1D-mollifier for $Rf$,  we obtain 
    \begin{equation*}
        |R(\cdot,\theta)|_{\TV} = \lim_{\epsilon \downarrow 0} 
        \left| Rf(\cdot,\theta) *_\sigma \rho_{1,\epsilon}(\cdot) \right|_{TV}= 
        \lim_{\epsilon \downarrow 0} 
        |R (f*\rho_\epsilon)(\cdot,\theta)|_{\TV} \,.
    \end{equation*}
    
Now, assume that the first of the two inequalities in \eqref{tt1} holds (the second inequality can be treated analogously). Then it follows that 
    \begin{equation*}
        |R (f*\rho_\epsilon)(\cdot,\theta)|_{\TV} 
        \leq  C_1 |(f*\rho_\epsilon)(\cdot,\theta)|_{\TV}\,,  
    \end{equation*}
and taking the limit $\lim_{\epsilon \downarrow 0}$ and then $\sup_\theta$ yields the assertion. 
\end{proof}

It is well known that $\BV(\R^n)$ can be embedded into $L^{\frac{n}{n-1}}(\Rn)$; cf.~\cite{Scherzbook}. Using this with  $n=1$, we obtain that for functions with compact support, 
    \begin{equation}\label{embed} 2 
        \|R f\|_{L^\infty(\R \times S^{n-1})} \leq \|Rf \|_{\no};
    \end{equation}
the fact that the embedding constant $C$ is $2$ is shown in Lemma~\ref{help1}). Furthermore, using rather standard estimates, it can be verified that (see \cite{KiHuarxiv} for a proof)
    \begin{equation}\label{comex} 
        \|f\|_{H^{1 - \frac{n}{2}-\epsilon}(\R^n)} \leq C \|Rf\|_{\no}\,, \qquad\forall \, \epsilon >0 \,.
    \end{equation}  

We now prove that the norm $\|Rf \|_{\no}$ can be estimated from above by the TV-norm. Even though this would follow from the Oberlin-Stein estimate, we have a simple proof which provides us with the explicit constant $C_2 =1$: 

\begin{proposition}\label{upper}
For $f \in \BV(\R^n)$ satisfying \eqref{supp}, there holds
    \begin{equation}
        \|Rf\|_{\no} \leq |f|_{\TV} \,. 
    \end{equation}
\end{proposition}
\begin{proof}
We start with a smooth function in $W^{1,1}(\R^n)$ and consider the parameterization of the hypersurface as in \eqref{para}: 
    \begin{align*}
        \|Rf\|_{\no} &= \sup_\theta \int_{\R} \left|\frac{d}{d\sigma} \int_{\R^{n-1}} f(\sigma \theta + t) \, dt \right| \, d\sigma =
        \sup_\theta \int_\R  \left|\int_{\R^{n-1}}  \nabla f(\sigma \theta + t) \cdot\theta \, dt \right| \, d\sigma \\
        &\leq 
        \sup_\theta \int_\R \int_{\R^{n-1}} \left|\nabla  f(\sigma \theta + t)\right| \, dt  \, d\sigma   = 
        \sup_\theta \int_{\R^n} |\nabla f(x) | \, dx   = \|\nabla f\|_{L^1} \,.
    \end{align*}
By Lemma~\ref{lemma2}, the result follows.
\end{proof}

We furthermore provide a lemma which is useful for showing the lower bounds. In the following, let $\ind{A}(x)$ denote the indicator function of $A$, i.e.,
    \begin{equation*}
        \ind{A}  = \begin{cases} 1 \,, & x \in A \,, \\ 
        0 \,, & \text{else} \,.  \end{cases}
    \end{equation*}
With this notation, we state the well-known coarea formula for $\BV$-functions \cite[Theorem 3.40]{Ambrosio}: For $f \in \BV$ we have
    \begin{equation}\label{coarea}
        |f|_{\TV} = \int_{-\infty}^\infty 
        | \ind{\{f > \alpha\}}|_{\TV} \, d\alpha \,.
    \end{equation}
A simple consequence is the following useful lemma: 

\begin{lemma}\label{help1} Let $f:\R \to \R^+_{0}$ be a nonnegative function with compact support, $|f|_{\TV} < \infty$, and with convex upper level sets. Then 
    \begin{equation*}
        |f|_{\TV} = 2\sup_{x} f(x) \,.
    \end{equation*}
On the other hand, if $f : \R \to \R$ is measurable, has compact support, and satisfies $|f|_{\TV} < \infty$, then
     \begin{equation*}
        |f|_{\TV} \geq  2\sup_{x} f(x) \,.
    \end{equation*}
\end{lemma}
\begin{proof}
Let $M$ be the supremum of $f$. Then the coarea formula \eqref{coarea} reads
    \begin{equation*}
        |f|_{\TV} =
        \int_{0}^{M}
        |\ind{\{ f > \alpha \}}|_{\TV} \, d \alpha \,.
    \end{equation*}
Now for $\alpha >0$ the set $\{ f >  \alpha \}$ is  a one-dimensional convex compact set, thus an interval $[a,b]$ or a degenerate interval $[a,a]$ (or empty for $\alpha > M$). The total variation of the indicator function of such (nonempty) intervals is $2$. Thus, the integrand in the coarea formula satisfies $|\ind{\{ f > \alpha \}}|_{\TV} = 2$ almost everywhere.

The second assertion can be shown by using the representation of $|\cdot|_{TV}$ by the essential pointwise variation \cite{Scherzbook}, which yields
    \begin{equation*}
        |f|_{TV} \geq \sum_{i=1}^N |f(x_i) -f(x_j)| \,,
    \end{equation*}
for an a.e.-equivalent representative of $f$ and any partition $x_i$. Now, taking $x_0 < x_1 < x_2 < x_3$ with $f(x_0) = f(x_3) = 0$ and $f(x_1) = \max(f^+)$, and $x_2 = \min(f^-)$ (or with $x_1$, $x_2$ reversed), we obtain that 
    \begin{equation*}
        |f|_{TV} \geq \max(f^+) + |\max(f^+) - \min(f^-)| + |\min(f^-)| \geq 2 \|f\|_\infty \,,
    \end{equation*}
which yields the assertion.
\end{proof}

\subsection{Lower bound for convex sets in the plane}

In this part, we verify $C |f|_{\TV} \leq \|Rf \|_{\no}$ and thus together with Proposition~\ref{upper} the norm equivalence $\TV$-$\no$ for indicator functions of convex sets in $\R^2$.

In two dimension, the Radon transform involves an integration over lines, 
    \begin{equation*}
        l(\sigma,\theta) = \{\sigma \theta + t \theta^\bot\, |\, 
        t \in \R\} \,.
    \end{equation*}
Let $K \subset \R^2$ be a bounded closed convex set with positive measure $\meas(K) >0$ in $\R^2$ and interior $\mathring{K}$. Let $\ind{K}(x)$ be the corresponding indicator function. By $\Pi_\theta(A)$, we denote the essential projection (cf.~\cite{Fusco}) of a set $A \subset \R^2$ on the line through the origin with direction $\theta$, i.e.,
    \begin{equation*}
        \Pi_\theta(A) = \{\sigma \in \R \,\vert\, 
        (R \ind{A})(\sigma,\theta) >0 \} \,,
    \end{equation*}
and by $\Gamma_\theta(A)$ the usual projection of $A$ onto the same line, i.e.,
    \begin{equation*}
        \Gamma_\theta(A) = \{ \sigma \in \R \,|\, 
        \exists t: \sigma \theta + t \theta^\bot \in A \} \,.
    \end{equation*}
For these sets, we denote by $|\Pi_\theta(A)|$ and $|\Gamma_\theta(A)|$ their one-dimensional Lebesgue measures. We have the following results.

\begin{lemma}\label{lemma6}
Let $K\subset \R^2$ be a bounded convex set with positive measure. Then, 
    \begin{align}
        \|R\ind{K}\|_{\no} &= 2\sup_{(\sigma,\theta)} 
        |R\ind{K}(\sigma,\theta)| = 2\diam(K) \,. \label{p2}
        \\
        |\ind{K}|_{\TV} &= \frac{1}{2} 
        \int_{\{(\sigma,\theta)\,:\, 
        R\ind{K}(\sigma,\theta) >0 \}} \, d\sigma \, d\theta\,.
        \label{p1}
    \end{align}
\end{lemma}
\begin{proof}
For the first identity \eqref{p2}, we observe that for a convex set, by convexity, the function $\sigma \to R \ind{K}(\sigma,\theta)$ also has  convex upper level sets, thus the first identity in \eqref{p2} follows by Lemma~\ref{help1}. Again by convexity, the diameter lies on a chord fully contained in $K$, and its length is given by $R\ind{K}(\sigma,\theta)$ for some $\sigma, \theta$. Since any other value of $R \ind{K}$ has a smaller value by definition of the diameter, we have that $\sup_{\sigma,\theta} R\ind{K}(\sigma,\theta) = \diam(K)$. 

For the second identity, we again observe that by convexity, $\Pi_\theta$ is a convex set in $\R$, and thus an interval. It follows from the proof of Theorem 3.2.35, p.~272 \cite{Federer} that the 1D-Lebesgue measure of $\Pi_\theta(K)$ equals that of $\Gamma_\theta(K)$. Thus, Cauchy's projection formula \cite[Theorem 3.2.15]{Federer} yields 
    \begin{equation*}
        \mathcal{H}^1(\partial K) = 
        \frac{1}{2}  \int_{\theta \in S^1} \int_{\Gamma_{\theta}(\mathring{K})} \, d\sigma \, d \theta 
        = \frac{1}{2}  \int_{\theta \in S^1} \int_{\Pi_{\theta}(\mathring{K})} \, d\sigma \, d \theta \,,
    \end{equation*}
where $\mathcal{H}^1$ denotes the one-dimensional Hausdorff measure. It was shown in \cite[Theorem 3.2.35]{Federer} that $\partial K$ is rectifiable, its boundary is Lipschitz and thus $|\ind{K}|_{\TV} = \mathcal{H}^1(\partial K)$; cf.~Proposition 3.62 in \cite{Ambrosio} and the remark at the end of its proof. Combining these arguments yields \eqref{p1}.
\end{proof} 

We can now establish the following equivalence of $\|\cdot\|_{\no}$ with $|\cdot|_{\TV}$:

\begin{theorem}\label{th:main1}
Let $K \subset \R^2$ be a bounded closed convex set with positive measure in $\R^2$, and let $f = \ind{K}$ denote the corresponding indicator function. Then, 
    \begin{align}\label{eq:main1}
        \frac{2}{\pi} |f|_{\TV} \leq \|R f\|_{\no} \leq 
        |f|_{\TV} \,.
    \end{align} 
\end{theorem}
\begin{proof}
First, we show that 
    \begin{align}\label{eded}
        \sup_{\theta \in S^1}   \int_{\{\sigma: R\ind{K}(\sigma,\theta) >0 \}} \, d\sigma = \diam(K)\,. 
    \end{align}
For this, we first prove that ``$\geq$'' holds in \eqref{eded}. The diameter $\diam(K)$ is the length of a chord contained in $K$. By elementary geometry and the definition of the diameter, the two lines $l(\sigma_1,\theta)$, $l(\sigma_2,\theta)$ through the endpoints of this chord and orthogonal to it are tangent to $K$. Thus, projecting orthogonal to this chord gives a set with length equal to $\diam(K)$. Hence, $\diam(K) = |\Gamma_{\theta}(K)| = |\Pi_{\theta}(K)|$ for some $\theta$ and where the equality of the Lebesgue-measure was already noted in the proof of Lemma~\ref{lemma6}. This means $\diam(K) = \int_{\{\sigma\,|\, R\ind{K}(\sigma,\theta) >0\}} \, ds$ for some $\theta$, which proves that ``$\geq$'' holds in \eqref{eded}.

Conversely, let $\theta_0$ be an $\epsilon$-supremum in \eqref{eded}, i.e., 
    \begin{equation*}
        \sup_{\theta \in S^1}   \int_{\{\sigma: R\ind{K}(\sigma,\theta) >0 \}} \, d\sigma  = 
        \int_{\{\sigma: R\ind{K}(\sigma,\theta_0) >0 \}} \, d\sigma   +\epsilon = 
        |\sigma_1-\sigma_2|   +\epsilon \,,
    \end{equation*}
where $\{\sigma: R\ind{K}(\sigma,\theta_0) >0 \} = |\sigma_1-\sigma_2| $ for some $\sigma_1,\sigma_2$. Note that by convexity, this set  must be an interval with boundary points $\sigma_1,\sigma_2$, which correspond to values on the line $\sigma \theta_0$ that are again tangent to $K$. Thus, there exist $x_1,x_2 \in \partial K$ such that $x_i = \sigma_i \theta_0 + t_i \theta_0^\bot$, $i=1,2$. By the Pythagorean theorem $ |\sigma_1-\sigma_2| \leq |x_1-x_2| \leq \diam(K)$. Hence, ``$\leq$'' holds in \eqref{eded}, and  consequently also  equality. 

Hence, combining \eqref{p2} and \eqref{p1} with Fubini's theorem, we obtain  
    \begin{align}\label{mean1}
    \begin{split}
        \abs{\ind{K}}_{TV}
        &\overset{\eqref{p1}}{=}\frac{1}{2} \int_{\theta \in S^1} \int_{\{\sigma: R\ind{K}(\sigma,\theta) >0 \}} \, d\sigma \, d\theta \leq 
        \frac{2\pi }{2} \sup_{\theta \in S^1}   \int_{\{\sigma: R\ind{K}(\sigma,\theta) >0 \}} \, ds \\
        \qquad & \leq \pi \diam(K) \overset{\eqref{p2}}{=} \frac{\pi}{2}  \|R\ind{K}\|_{\no} \,,
    \end{split}
    \end{align}
which yields the assertion.
\end{proof}

\begin{remark} \rm 
From the proof of Theorem~\ref{th:main1} we may also conclude the well-known result that the ratio of diameter over perimeter, $\frac{\diam(K)}{|\chi_{K}|_{TV}}$, is between $\frac{1}{2\pi}$ and $1$. This has been named the Rosenthal-Szasz theorem \cite{Rosz,RoSzOrig,Bla}. From these references, one can also find examples when the inequalities are sharp. The lower bound $\frac{2}{\pi} |f|_{\TV} = \|R f\|_{\no}$ is achieved for the indicator function of $K$ being a ball, but also for any convex set of constant width (e.g., the Reuleaux triangle) (and only for them). This follows from the study of when the first inequality in \eqref{mean1} is sharp, namely when the integrand does not depend on $\theta$ (which means that the support of $Rf$ is independent of $\theta$ which means objects of constant width). Equality in the upper bound in \eqref{eq:main1}, i.e., $ |f|_{\TV} =  \|R f\|_{\no} = 2 \diam(K)$ is not obtained by convex sets satisfying our assumptions, since  by the triangle inequality, twice the diameter is always strictly less than the perimeter. However, for thin rectangular sets $K$ with side lengths $1$ and small $\epsilon$, the upper bound for $\|R f\|_{\no}$ is approached in the limit $\epsilon \to 0$ when the rectangle degenerates into a line. 
\end{remark}

Let us now give some generalizations: The norm equivalence in \eqref{eq:main1} clearly also holds when $f$ satisfying \eqref{supp} is restricted to a finite-dimensional subspace, since both expression define norms, and in finite dimensions all norms are equivalent. The corresponding constant depends on the dimension. In the following special case, we explicitly find the dimensional dependence. 

\begin{proposition}\label{proporp}
Let $f$ be a linear combinations of indicator functions, i.e., 
    \begin{equation*}
        f = \sum_{i=1}^{N} a_i \chi_{K_i} \,,
    \end{equation*}
where $K_i \subset \R^2$ are bounded, closed, convex sets with positive measures, and the coefficients $a_i \geq 0$ are nonnegative. Then, there holds 
    \begin{equation*}
        \frac{2}{\pi N} |f|_{\TV} \leq  \|Rf\|_{\no} \,.
    \end{equation*}
\end{proposition}
\begin{proof}
The triangle inequality and the previous results for convex sets yield
    \begin{align*}
        |f|_{\TV} \leq \sum_{i=1}^{N} |a_i| |\ind{K_i}|_{\TV} \leq \sum_{i=1}^{N} |a_i| \pi \, \diam(K_i)   \leq 
        \pi {N} \max_i {a_i \diam(K_i)} \,.
    \end{align*}
Now, take $\sigma_*,\theta_*$ corresponding to a line $l(\sigma_*,\theta_*)$ through the diameter of a $K_i$ which satisfies $a_i \diam(K_i) = \max_i a_i \diam(K_i)$. Then,  the integral $Rf(\sigma_*,\theta_*)$  taken over this diameter, by the positivity of $f$, must be larger than $a_i \diam(K_i)$. Thus, by \eqref{embed}, $\abs{Rf(\sigma_*,\theta_*)} \leq \|R f\|_\infty \leq \frac{1}{2} \|Rf\|_{\no}$, and the result follows. 
\end{proof}

As a special case of the above results, we have the following:

\begin{corollary}\label{maincor}
Let $f:\R^2 \to \R$ be a nonnegative function taking only $N$ different values, and let the corresponding $N$ upper level sets be closed, convex, and bounded with positive measure. Then 
    \begin{equation*}
        \frac{2}{\pi N} |f|_{\TV} \leq  \|Rf\|_{\no} \,.
    \end{equation*}
\end{corollary}
\begin{proof}
This holds because we may write $f = \sum_{i=1}^{N} a_i \chi_{K_i}$ with $K_i$ and $a_i>0$ satisfying the assumptions of Proposition~\ref{proporp}.
\end{proof}

\begin{remark} \rm 
Note that for nonnegative radially symmetric functions with convex upper level sets, one can obtain $C |f|_{\TV} \leq \|Rf\|_{\no}$, where the constant $C$ does not depend on $N$. A more general result is proven below in Theorem~\ref{th4}.
\end{remark}

\begin{remark} 
\rm 
The convexity assumption cannot be dropped in the previous estimates, as the following example of an annulus shows. Consider the indicator function of an annulus with outer radius $r_1 = 1$ and inner radius $r_2 = 1 -\epsilon$, with $\epsilon>0$ small. The total variation of the indicator function is the perimeter, which equals $2 \pi + 2 \pi (1-\epsilon)$ and which tends to $4 \pi$ as $\epsilon \to 0$. The Radon transform of this radially symmetric function is independent of $\theta$ and can be calculated by \eqref{radsqu} (as the Radon transform of the indicator function of the outer ball minus the inner ball). It turns out that in the offset variable $\sigma$, $Rf(\sigma,\theta)$ has a double peak shape with two maxima with value $R_{max}$ and one local minimum with value $R_{min}$ in between, and the total variation can be calculated as $2 R_{max} + 2(R_{max} - R_{min})$. The value of the maxima is the length of the cord touching the inner ball and thus given by $2 \sqrt{1-r_2^2}$, while the inner minimum corresponds to the length of the line through the center, i.e., $2 (1- r_2)$. Thus, 
    \begin{equation*}
        \|Rf\|_{\no} = 8 \sqrt{1-r_2^2} - 4 (1- r_2)
        = 8 \sqrt{1-(1-\epsilon)^2}  - 4 \epsilon 
        = O(\sqrt{\epsilon}) \,,
    \end{equation*}
and this value tends to $0$ as $\epsilon \to 0$. Hence, in this nonconvex case, we cannot have a constant $C$ with $|f|_{TV} \leq C \|R f\|_{\no}$.
\end{remark}

{
\begin{remark} \rm
It is interesting to study a special case for the lower bound in Proposition~\ref{proporp}: Consider the indicator function $f$ of $N$ disjoint balls in the plane with small radius $\epsilon$ contained in the unit ball. What would be an arrangement of the centers of these balls such that $\sup_{(\sigma,\theta)} Rf(\sigma,\theta)$ is extremal. The maximal value seems to be the case when the centers are arranged in a line, yielding an $N$-independent constant in Proposition~\ref{proporp}. (This would result in the convenient case of the $\no$-norm being equivalent to the $\TV$-norm). The opposite (and for us inconvenient case) of an arrangement of the centers having minimal $\|Rf\|_\infty$-norm seems to lead to a difficult combinatorial problem, namely essentially to that of arranging balls such that any line in the plane intersects them in a minimal number. An arrangement of the centers along a circle (shaping something like an annulus) seems to be at least an approximate minimal solution, but any solid theoretical statement about this problems is out of reach for us.
\end{remark}
}

\subsection{Steiner Symmetrization in \boldmath $\R^2$}

In this section, we derive TV-RTV norm equivalences based on Steiner symmetrization. For this, let $\Omega \subset \R^2$ be a measurable set, and let $l(0,\theta)$, as above, be a line in $\R^2$ through the origin and with direction $\theta^\bot$. Then by $\Omega^\theta$ we denote the Steiner symmetrization of $\Omega$ around that line; see, e.g., \cite{Henrot}. Recall that this is the uniquely defined set which is symmetric with respect to $l(0,\theta)$ such that any intersection with an line $l' = l(\sigma,\theta^\bot)$ orthogonal to $l(0,\theta)$ is given by an open interval $(-a,a)$ centered at $l(0,\theta)$ and with length equal to the length of the intersection of $\Omega$ and $l'$. The length is exactly $R\ind{\Omega}(\sigma,\theta^\bot)$. 
Thus, 
    \begin{equation*}
        \Omega^\theta = \left\{x = t \theta +\sigma \theta^\bot  \,\big|\, -\frac{1}{2}R\ind{\Omega}(\sigma,\theta^\bot) < t
        <  \frac{1}{2}{R\ind{\Omega}(\sigma,\theta^\bot)} \right\} \,,
    \end{equation*} 
In particular, integrals orthogonal to the symmetry axis are identical for $\Omega$ and $\Omega^\theta$, i.e., we have the identity 
    \begin{equation}\label{eq} 
        R \ind{\Omega^\theta}(\sigma,\theta^\bot) = R\ind{\Omega}(\sigma,\theta^\bot) \,, 
    \end{equation} 
and consequently,
    \begin{equation*}
        |R \ind{\Omega}(\cdot,\theta^\bot) |_{\TV} = |R \ind{\Omega^\theta}(\cdot,\theta^\bot) |_{\TV} \,.
    \end{equation*}
For a nonnegative function $f$, the corresponding Steiner symmetrization $f^\theta$ is defined analogously by symmetrizing its upper level sets, i.e.,  $f^\theta$ is uniquely defined by 
    \begin{equation*}
        \{ f^\theta >t\} = \{ f >t\}^\theta \,.
    \end{equation*}
Hence, similarly to \eqref{eq} we obtain that $Rf^\theta(\sigma,\theta^\bot) =  Rf(\sigma,\theta^\bot)$, which yields the following refinement of the upper bound \eqref{upper}:  
    \begin{equation} 
        \|R f \|_{\no} \leq \sup_{\theta} |f^\theta|_{\TV} \,, \qquad \forall \, f \geq 0 \,.  
    \end{equation}
Note that Steiner symmetrization reduces the perimeter (see, e.g., \cite[Theorem~1.1]{Fusco}), i.e., $ |f^\theta|_{\TV} \leq  |f|_{\TV}$. Using this symmetrization, we are able to establish yet another estimate for the $\no$-norm in the next proposition.  

\begin{proposition}
Let $\theta \in S^1$ and let $\Pi_{\theta^\bot}(\Omega)$ be the essential projection of $\Omega \subset \R^2$ onto the line
$l(0,\theta^\bot)$. Then, there holds 
    \begin{equation} \label{auchwas}
    \begin{split}
        & \frac{1}{\sqrt{2}} \left(  \frac{1}{2} |R \ind{\Omega}(\cdot,\theta^\bot) |_{\TV} + \int_{\Pi_{\theta^\bot}(\Omega)} d\sigma \right) \\
        & \qquad \leq 
        \frac{1}{2}  |\ind{\Omega^\theta}|_{\TV}  \leq   \frac{1}{2} |R \ind{\Omega}(\cdot,\theta^\bot) |_{\TV}  + \int_{\Pi_{\theta^\bot}(\Omega)} \, d\sigma \,. 
    \end{split} 
    \end{equation} 
\end{proposition}
\begin{proof}
For illustration, we first prove this result for the case that $R\ind{\Omega}(\sigma,\theta^\bot)$ is continuous and $\partial \Omega^\theta$ is Lipschitz. Furthermore, assume first that $\ol{\Pi_{\theta^\bot}(\Omega)}$ is connected, i.e., an interval. The set $\Omega^\theta$ is given by 
    \begin{equation*}
        \Kl{ x = t\theta +\sigma \theta^\bot \,|\,  \sigma \in \Pi_{\theta^\bot}(\Omega)  \,,  |t|  < \frac{R\ind{\Omega}(\sigma,\theta^\bot)}{2}} \,.
    \end{equation*}
Thus, the boundary of the symmetrized set can be represented as a graph of a continuous function. Its perimeter, i.e., $|\ind{\Omega^\theta}|_{\TV}$, is then twice the arclength of the boundary curve $\sigma\to (\sigma, \frac{R\ind{\Omega}(\sigma,\theta^\bot)}{2})$; see e.g.~\cite[Thm.~2.6.2]{Burago}. In the following, let $\mathcal{P}(\overline{\Pi_{\theta^\bot}(\Omega)})$ denote a partition of $\overline{\Pi_{\theta^\bot}(\Omega)}) \subset\R$, i.e, a vector $(\sigma_i)_{i=1}^N$, $n\in\N$, with $\sigma_i < \sigma_{i+1}$. Then,  
    \begin{align*}  &\frac{1}{2} 
        |\ind{\Omega^\theta}|_{\TV} \\
        & \ =
        \sup_{(\sigma_i)_{i=1}^N \in \mathcal{P}(\overline{\Pi_{\theta^\bot}(\Omega)})} 
        \sum_{i=1}^N \sqrt{\left|\frac{1}{2} (R\ind{\Omega}(\sigma_i,\theta^\bot) - \frac{1}{2}R\ind{\Omega}(\sigma_{i-1},\theta^\bot) \right|^2 + |\sigma_i-\sigma_{i-1}|^2} \,. 
    \end{align*}
Now let 
    \begin{equation*}
         \Var R\ind{\Omega}(\cdot,\theta^\bot) :=  
         \sup_{(\sigma_i)_{i=1}^N \in \mathcal{P}(\overline{\Pi_{\theta^\bot}(\Omega)})}
         \sum_{i=1}^N 
        \left|(R\ind{\Omega}(\sigma_i,\theta^\bot) - R\ind{\Omega}(\sigma_{i-1},\theta^\bot) \right|  \,.
    \end{equation*}
Using $\frac{1}{\sqrt{2}} (|a| + |b|) \leq \sqrt{a^2+b^2} \leq |a| + |b| $, we thus find that 
    \begin{align*} 
        &\frac{1}{\sqrt{2}} \left( \frac{1}{2} 
        \Var R\ind{\Omega}(\cdot,\theta^\bot)  +  
       \sup_{(\sigma_i)_{i=1}^N \in \mathcal{P}(\overline{\Pi_{\theta^\bot}(\Omega)})}
         \sum_{i=1}^N 
        |\sigma_i-\sigma_{i-1}|  \right) \leq  \frac{1}{2} |\ind{\Omega^\theta}|_{\TV} \\
        & \leq 
        \frac{1}{2}  \Var R\ind{\Omega}(\cdot,\theta^\bot)  +  
        \sup_{(\sigma_i)_{i=1}^N \in \mathcal{P}(\overline{\Pi_{\theta^\bot}(\Omega)})}
        \sum_{i=1}^N |\sigma_i-\sigma_{i-1}| \,.
    \end{align*}
Since we assumed that $R\ind{\Omega}$ is continuous, it follows that 
	\begin{equation*}
		\Var R\ind{\Omega}(\cdot,\theta^\bot) = |R\ind{\Omega}(\cdot,\theta^\bot)|_{\TV} \,,
	\end{equation*}
which together with 
	\begin{equation*}
        \sup_{(\sigma_i)_{i=1}^N \in \mathcal{P}(\overline{\Pi_{\theta^\bot}(\Omega)})}
        \sum_{i=1}^N  
        |\sigma_i-\sigma_{i-1}| = | \overline{\Pi_{\theta^\bot}(\Omega) }|
		= \int_{\Pi_{\theta^\bot}(\Omega)|} d\sigma 
	\end{equation*}
yields the assertion. In case that $\ol{\Pi_{\theta^\bot}(\Omega)}$ has multiple components separated by intervals where $R\chi_\Omega$ vanishes, we may apply the previous argument component-wise and get the analogous result, since $|\cdot|_{\TV}$ is additive over these components.

For the general case, see \cite[p.~162, line~9]{Fusco}, which implies
    \begin{equation*}
        |\ind{\Omega^\theta}|_{\TV}  \leq  |  R\ind{\Omega}(\cdot,\theta^\bot) |_{\TV} + 2 \int_{\Pi_{\theta}(\Omega)} \, d\sigma\,.
    \end{equation*}
The lower bound can be proven using similarly arguments as in \cite{Fusco}. 
\end{proof}

The above proposition sheds light on the meaning of the $\no$-norm: It is equivalent to  the perimeter of a symmetrized set up to a term related to the width of the set. By the isoperimetric inequality, a lower bound for $|\ind{\Omega^\theta}|_{\TV}$ is given by the TV-norm of the spherical symmetric rearrangement of $\Omega$. Thus, we have the following lower bound: 

\begin{theorem}
Let $\Omega$ be a measurable set in $\R^2$, and let $\Omega^*$ be its spherical symmetric rearrangement. Then there holds
    \begin{equation}
        |\ind{\Omega^*}|_{\TV} \leq  \frac{1}{2} \|R \ind\Omega\|_{\no} + \inf_\theta 
        \int_{\Pi_{\theta} (\Omega)} \, d\sigma \,.
    \end{equation}
Additionally, let $\partial \Omega$ be Lipschitz continuous.  
Then we have the lower bound 
    \begin{equation}\label{ssyy}  
        |\ind{\Omega^*}|_{\TV}- \frac{1}{\pi} 
        |\ind{\Omega}|_{\TV} \leq  \frac{1}{2} \|R \ind\Omega\|_{\no} \,.
    \end{equation}
\end{theorem}
\begin{proof}
The first estimate follows from \eqref{auchwas} using $|\ind{\Omega^*}|_{\TV}\leq |\ind{\Omega^\theta}|_{\TV}$ and then taking the infimum over $\theta$. The second bound follows from the Crofton formula \cite[Theorem 3.2.26, 2.10.15]{Federer}. Let $n(\partial \Omega,l(\sigma,\theta))$ be the number of intersections of the line $l(\sigma,\theta)$ with $\partial \Omega$. The set of $\sigma$ where this number is nonzero contains the set of $\sigma$ where  $R_\Omega(\sigma,\theta)>0$, and on this set, $n(\partial \Omega,l(\sigma,\theta))$ must be larger than or equal to  $2$, i.e.,
    \begin{align*} 
        |\ind{\Omega}|_{\TV} &= 
        \mathcal{H}^1(\partial \Omega) =
        \frac{1}{4} \int_{S^1} \int_\R n(\partial \Omega,l(\sigma,\theta) \, d\sigma \, d\theta \geq 
        \frac{1}{4}  \int_{S^1} 
        \int_{R(\sigma,\theta)>0} 2 \, d\sigma 
        \\ &
        \geq \frac{2}{4} 2 \pi \inf_\theta   \int_{\Pi_{\theta}(\Omega) } \, d\sigma   
        = \pi \inf_\theta   \int_{\Pi_{\theta}(\Omega) } \, d\sigma \,. 
    \end{align*}
\end{proof}

The estimate \eqref{ssyy} is, of course, only relevant when the left-hand side is positive, that is, for sets $\Omega$ not too far off from a spherical shape. Note that this estimate does not require convexity. Finally, we show a lower bound for nonnegative functions with convex upper level sets and two axis of symmetry. 

\begin{theorem}\label{th4}
Let $f : \R^2 \to \R$ be nonnegative with convex upper level sets and having two (orthogonal) axes of symmetry. Then there is a constant $C$ such that 
    \begin{equation*}
        C |f|_{\TV}  \leq \|R f\|_{\no} \,.
    \end{equation*}
\end{theorem}
\begin{proof}
Without loss of generality, we assume the $x$ and $y$ axis to be axes of symmetry, i.e., $f(x,y) = f(-x,y)$ and $f(x,y) = f(x,-y)$. Let $\theta = (0,1)$ be  the $y$-direction and $\theta^\bot$ the $x$-direction. We have $ Rf(y,\theta) = \int_{\R} f(x,y) dx $ and $Rf(x,\theta^\bot) = \int_{\R}  f(x,y)\,dy$. Note that $Rf(y,\theta)$ and $Rf(x,\theta^\bot)$ as functions of $y$ and $x$, respectively, have (by convexity and symmetry) convex centered level sets, thus intervals. In particular, $0$ is contained in all nonempty level sets, thus 
    \begin{align}\label{ti} 
    \begin{split}
        \sup_x Rf(x,\theta^\bot) &= 
        Rf(0,\theta^\bot) = 
        \int f(0,y) \, dy \,, \\ 
        \sup_y Rf(y,\theta) &= Rf(0,\theta) = \int f(x,0) \, dx \,. 
    \end{split}
    \end{align}
For each upper level set 
$\Omega_\alpha = \{f> \alpha\}$, we obtain from \eqref{auchwas} that 
     \begin{equation*}
        |\ind{\Omega_\alpha}|_{\TV} \leq 
        |R\ind{\Omega_\alpha}(\cdot,\theta^\bot)|_{\TV} 
        + 2 |\Pi_{\theta^\bot} \Omega_\alpha| \,, 
    \end{equation*}
and the analogous result for $|R\ind{\Omega_\alpha}(\cdot,\theta)|_{\TV}$. By convexity and symmetry, \eqref{p1} and \eqref{ti}, it follows that 
    \begin{equation*}
        |R\ind{\Omega_\alpha}(\cdot,\theta^\bot)|_{\TV} 
        = 2\int \ind{\Omega_\alpha}(0,y) \, dy \,,
        \qquad |R\ind{\Omega_\alpha}(\cdot,\theta)|_{\TV} 
        = 2\int \ind{\Omega_\alpha}(x,0) \, dx \,.
    \end{equation*}
We now show that
    \begin{equation*}
        |\Pi_{\theta} \Omega_\alpha|  =
        \left|\left\{y|R\ind{\Omega_\alpha}(y,\theta) >0\right\}\right| \leq \int \ind{\Omega_\alpha}(0,y) \, dy \,. 
    \end{equation*}
Define $I_y := \{y|R\ind{\Omega_\alpha}(y,\theta) >0\}$ and $I_{y,\epsilon} := \{y|R\ind{\Omega_\alpha}(y,\theta) >\epsilon\}$. By convexity and symmetry, $I_y$ and $I_{y,\epsilon} $ are intervals centered at $0$. Clearly, $I_y = \bigcup_{\epsilon>0} I_{y,\epsilon}$, and since the $I_{y,\epsilon}$ are ordered, $|I_y| = \lim_{\epsilon \to 0}  |I_{y,\epsilon}|$. Let $\pm a_{\epsilon}$ be the boundary points such that $I_{y,\epsilon} = [a_{\epsilon},-a_{\epsilon}]$ (without loss of generality we assume this a closed interval; if not we reduce $a_{\epsilon}$ by a small $\delta$). Since $R\ind{\Omega_\alpha}(\pm a_{\epsilon},\theta)>\epsilon$, the lines $\{(x,\pm a_\epsilon)\}$, $|x| \leq \frac{\epsilon}{2}$ must be contained in $\Omega_\alpha$, and by convexity, the rectangle $S:= \{(x,y)| |x| \leq \frac{\epsilon}{2}, |y| \leq a_\epsilon\}$ is contained in $\Omega_\alpha$. Hence, 
    \begin{equation*}
        \int \ind{\Omega_\alpha}(0,y) \, dy  
        \geq \int \ind{S}(0,y) dy = 2|a_\epsilon| 
        = |I_{y,\epsilon}|\,.
    \end{equation*}
Thus, taking $\epsilon \to 0$ gives
    \begin{equation*}
        |\Pi_{\theta} \Omega_\alpha| 
        \leq \int \ind{\Omega_\alpha}(0,y) \,dy\,,
    \end{equation*}
which yields 
    \begin{equation*}
        |\ind{\Omega_\alpha}|_{\TV}
        \leq  2\int \ind{\Omega_\alpha}(x,0) dx
        + 2 \int \ind{\Omega_\alpha}(0,y) dy \,. 
    \end{equation*}
Integrating over $\alpha$, using the coarea formula and Fubini's theorem yields 
    \begin{align*}
        |f|_{\TV} 
        &= \int_{0}^\infty 
        |\ind{\Omega_\alpha}|_{\TV} \, d \alpha \\ 
        &\leq  2\int_{\R} \int_{0}^\infty   \ind{\Omega_\alpha}(x,0) \, d\alpha \, dx +
        2 \int_{\R}  \int_{0}^\infty  \ind{\Omega_\alpha}(0,y)  \, d\alpha \, dy \,,
    \end{align*}
and by the so-called layer-cake representation, it follows that 
    \begin{align*} 
        \int_{\R} \int_{0}^\infty \ind{\Omega_\alpha}(0,y) \, d\alpha  = \int_{\R} f(0,y) \, dy \,, \quad 
        \int_{\R} \int_{0}^\infty  \ind{\Omega_\alpha}(x,0) \, d\alpha  \, dx = 
        \int_{\R} f(x,0) \, dx \,. 
    \end{align*}
Combining this with \eqref{ti} and \eqref{p1} now yields the assertion. 
\end{proof}

\subsection{Conditional stability for smooth functions in \boldmath $\R^2$}

This section is devoted to lower bounds for $\|Rf\|_{\no}$ with respect to the $\TV$-norm, again in two dimensions, and in case that $f$ has additional smoothness and is nonnegative. The results are achieved by Gargliardo-Nirenberg estimates. In the following, we denote by $W^{s,1}(\R^n)$, for $s \in \R^+$, the Sobolev spaces with differentiation order $s$ based on $L^1$. For integer $s$, this is just the sum of the $L^1$-norms of all derivatives up to order $s$. For noninteger $s = \floor{s}  + \beta$, where $\floor{s}$ is an integer and $\beta \in (0,1)$, the norm of $f$ is defined as the $W^{\floor{s},1}$-norm plus the seminorm, cf.~\cite{Adams,Brezis},
    \begin{equation*}
        \int_{\R^n} \int_{\R^n} 
        \frac{|f(x)-f(y)|}{|x-y|^{n + \beta}} \, dx \, dy\,.
    \end{equation*}
For two-dimensional functions $f$, let $f_{\sigma,\theta}(t):=  f(\sigma \theta + t \theta^\bot)$ denote the restriction of $f$ to the line $l(\sigma,\theta)$. From \cite[Proposition 2.6]{Brezis}, it follows that 
    \begin{equation*}
        \int_{\theta \in S^1} 
        \int_{\R} \| f_{\sigma,\theta}\|_{W^{s,1}(\R)} \,
        d\sigma \, d\theta 
        \leq \|f\|_{W^{s,1}(\R^2)}\,. 
    \end{equation*}
Note that the above inequality can be sharpened into equality up to a constant. Furthermore, it may also be extended to \TV-functions in the form, cf.~\cite{Kindermann},
    \begin{equation}\label{K1}
        |f|_{\TV} = \frac{1}{4} \int_{\theta \in S^1} \int_{\R} |f_{\sigma,\theta}|_{\TV}  \, d\sigma 
        d\theta\,.
    \end{equation}
This then allows us to derive the following result: 

\begin{theorem}
Let $f\geq 0 $ be a nonnegative function satisfying the support condition \eqref{supp} and assume that there exists a constant $M > 0$ such that 
    \begin{equation*}
        \|f\|_{W^{s,1}(\R^2)} \leq M \,,
        \qquad \text{ for some }
        s \in (1,\infty) \,.
    \end{equation*}
Then there exist a constant $C > 0$ such that
    \begin{equation*}
        M^{-\frac{1}{s-1}} |f|_{\TV}^\frac{s}{s-1} 
        \leq C \|Rf\|_{\no} \,.
    \end{equation*}
\end{theorem}
\begin{proof}
For a one-dimensional function $g(t)$ and any $s \in (1,\infty)$, the following Gagliardo-Nirenberg estimate holds: \cite[Theorem 1]{Brezis} for $s \in (1,\infty)$:
    \begin{equation*}
         \|g\|_{W^{1,1}(\R)} \leq  \|g\|_{L^1(\R)}^\rho\|g\|_{W^{s,1}(\R)}^{1-\rho} \,, \qquad  \rho =  1- \frac{1}{s} \in (0,1) \,.
    \end{equation*}
Applying this to the choice $g(t):= f_{\sigma,\theta}(t)$ for fixed $\sigma$ and $\theta$, we obtain
    \begin{equation*}
        \int_\R \left|\frac{d}{dt}  f_{\sigma,\theta}(t) \right| \, dt  
        \leq \left(\int_\R |f_{\sigma,\theta}|  \, dt \right)^\rho
        \| f_{\sigma,\theta}\|_{W^{s,1}(\R)}^{1-\rho} \,.
    \end{equation*}
Since $f$ is assumed to be nonnegative, it follows that $\int |f_{\sigma,\theta}| \, dt = \int f_{\sigma,\theta} \, dt = 
Rf(\sigma,\theta)$. Hence, integrating over $\sigma$ and $\theta$ and using \eqref{K1} and \eqref{embed} yields 
    \begin{align*} 
        4|f|_{\TV} &\overset{\eqref{K1}}{=} 
        \int_{\theta \in S^1} \int_{\R} 
        \int_{\R}  \left|\frac{d}{dt}  f_{\sigma,\theta}(t) \right| \, dt \, d\sigma \,d\theta \\ 
        &\leq 
        \int_{\theta \in S^1} 
        \int_{\R}|Rf(\sigma,\theta)|^{\rho}  \| f_{\sigma,\theta}\|_{W^{s,1}(\R)}^{1-\rho} \, d\sigma \, d\theta \\ 
        & \leq 
        (\sup_{\sigma,\theta} |Rf(\sigma,\theta)| )^\rho 
        \int_{\theta \in S^1} 
        \int_{\R} \| f_{\sigma,\theta}\|_{W^{s,1}(\R)}^{1-\rho}  \,d\sigma \,d\theta \\
        & \overset{\eqref{embed}}{\leq} C \|R f\|_{\no}^\rho 
        \left(\int_{\theta \in S^1} 
        \int_{\R} \| f_{\sigma,\theta}\|_{W^{s,1}(\R)} \, d\sigma \, d\theta \right)^{1-\rho} \,,
	\end{align*}
where we have used the H\"older inequality and the fact that $f_{\sigma,\theta}$ has compact support in $[-1,1] \times [0, 2\pi]$, which concludes the proof. 
\end{proof}

Thus, if $f$ in this theorem is of high smoothness $s \gg 1$,  
then the exponent tends to $1$, and thus the norm equivalence between $\TV$ and $\no$ holds asymptotically as the smoothness tends to $\infty$ for all such nonnegative functions.

\section{Filtering and Regularization}\label{sec:4}

As stated in the introduction, one motivation for the use of different sinogram norms $\| \cdot \|_S$ is to obtain novel nonlinear filtered backprojection formulas for the approximate inversion of the Radon transform. 

For this, recall the Tikhonov functional \eqref{Radon_Tikhonov}. As noted in the introduction, if the penalty $\|f\|$ is equivalent to one of the sinogram norms $\|Rf \|_S$ defined above, it is reasonable to replace $\|f\|$ by  $\|Rf \|_S$. Instead of \eqref{Radon_Tikhonov}, we then minimize 
    \begin{equation}\label{eqin} 
        T_S(f):= \frac{1}{2} \|R f - y\|^2 + \alpha \|R f\|_S \,,
    \end{equation}
for $f$, and thus, a regularized solution is given by $f_\alpha = \argmin_f T_S(f)$. However, observe that minimizing the Tikhonov functional over $f$ is (roughly) equivalent to minimizing over $Rf$. More precisely, upon setting $g = Rf$, we can rewrite this problem as the constrained minimization problem
    \begin{equation}\label{kabel}
        \min_{g} 
        \|g - y\|^2 + \alpha \|g\|_S \,, \qquad 
        \text{ subject to } g = Rf \,.
    \end{equation}
Without the constraints $g = Rf$, we recognize that a minimizer $g$ is obtained by applying the well-known proximal operator 
    \begin{equation}\label{babel}
        \prox_{\alpha \|\cdot\|_S}(y) : =  
        \argmin \Kl{\frac{1}{2} \| z-y\|^2  + \alpha \|z\|_S} \,,
    \end{equation}
which, at least formally, leads to the following solution formula for $f_\alpha$:
    \begin{equation}\label{fffp} 
        R f_\alpha = \prox_{\alpha \|\cdot\|_S}(y) \,. 
    \end{equation}
The above derivation is informal, since a solution of \eqref{fffp} does not need to exist, since the proximal map is not necessarily in the range of $R$; see also Remark~\ref{remark_range}. A related issue is that we cannot guarantee in all cases that \eqref{eqin} has a minimizer, since
sets of functions $f$ with bounded $\|R f \|_S$ -norm are not necessarily bounded 
in $L^2$ (but only in the weaker space 
$H^{-\epsilon}$; cf.~\eqref{comex}), and thus,  
a classic existence proof using weak compactness 
is not feasible. 
However, in this section, we do not attempt to achieve a theoretically based algorithm, but rather investigate practical consequences of our stability estimates. Of course, from a 
practical point of view, in 
a discretized setting of numerical computation, the existence of a minimizer is clear.

Next, note that in \eqref{fffp} the proximal operator is independent of the Radon transform, and the problem is now separated into an initial nonlinear filtering of the data, followed by an inversion of the filtered data. Based on these ideas, we proposed a novel nonlinear backprojection method by applying simple inversion formulas for the Radon transform to the filtered data $\prox_{\alpha \|\cdot\|_S}(y)$, i.e., 
    \begin{equation}\label{appfil} 
        \text{`` $f \sim R^{-1}\prox_{\alpha \|\cdot\|_S}(y) $ ''} \,,
    \end{equation}
or, since the inverse may not be well-defined, by using an approximate inverse:
    \begin{equation}\label{bappfil} 
        f\sim R_\alpha^\# \prox_{\alpha \|\cdot\|_S}(y) \,,
    \end{equation}
where $R_\alpha^\# $ is a regularized approximation of $R^{-1}$. In our numerical experiments, we use simple filtered backprojection for $R_{\alpha}^{\#}$ via Matlab's ``iradon'' command:
    \begin{equation}\label{irbb}
        f_\alpha = \text{iradon}(\prox_{\alpha \|\cdot\|_S}(y)) \,.
    \end{equation}
However, note that the nonlinear sinogram filtering can, of course, in principle be combined with any other regularization method in place of \text{iradon}. For the norm $\norm{\cdot}_S$ in the proximal operator $\prox_{\alpha \|\cdot\|_S}$, we either use the Radon-TV norm $\|\cdot\|_{\no}$ or the $p$-th power of the Sharafutdinov norm, $\norm{\cdot}_\RsptSnR^p$. 

\begin{remark} \rm
Even though a full analysis of the method given in \eqref{bappfil} is outside the scope of the article, we now briefly sketch how a possible convergence analysis of it could proceed. Based on standard theory, the computed $f$ in \eqref{bappfil} would converge to the exact solution, if $\prox_{\alpha \|\cdot\|_S}(y)$ is treated as perturbed data and the regularization parameter in $R_\alpha^\#$ is chosen according to $\delta$, where here $\delta$ is the bound $\| \prox_{\alpha \|\cdot\|_S}(y) - R f^\dagger\| \leq C \delta$. Similarly as for Tikhonov regularization, we may then estimate \eqref{babel} to verify that this condition can be satisfied if the corresponding regularization parameter $\alpha$ is also chosen appropriately. 
\end{remark}

\begin{remark}\label{remark_range} \rm 
Since the proposed method \eqref{irbb}, as well as \eqref{bappfil}, is only an approximation of the motivating Tikhonov regularization \eqref{eqin}, it may be of interest to illustrate ways to modify the algorithm such that it becomes an exact minimization method for \eqref{eqin}. Ignoring technical difficulties (such as whether the subgradient of $\|.\|_S$, denoted by $\partial_{\|.\|_S}$, satisfies the condition $\partial \|R f\|_S = R^*(\partial_{\|R f\|_S})$; cf., e.g., \cite[Prop.~5.7]{EkTe}), the optimality condition of \eqref{eqin} reads 
    \begin{equation*}
        R^*\left[ R f - y + \alpha \partial_{\|. \|_S}(R f)\right] = 0 \quad \Longleftrightarrow \quad R f + \alpha \partial_{\|R f\|_S} = y + \eta \,.    
    \end{equation*}
Here, $\eta$ is an element in the nullspace of $R^*$, which equals the orthogonal complement of $\overline{\text{Ran}(R)}$, i.e., the closure of the range of $R$. (Note that our method above is obtained by setting $\eta$ to $0$.) Now, let $Q_R$ denote the orthogonal projector onto $\overline{\text{Ran}(R)}$. Then, the optimality 
conditions are equivalently to
    \begin{equation*}
        (I - Q_R)[g + \alpha \partial_{\|g\|_S} - y] = 0 \,, \qquad g = Rf \,.
    \end{equation*}
These are also the optimality conditions of the proximal functional \eqref{kabel}, if the constraint $g \in \overline{\text{Ran}(R)}$ is included via an 
indicator function, i.e., 
    \begin{equation*}
        \|g - y\|^2 + \alpha \|g\|_S  + \chi_{\overline{\text{Ran}(R)}}(g) \,.    
    \end{equation*}
Since this functional is a sum of two convex functionals (with respect to $g$), splitting methods are a natural choice for its minimzation. For instance, using a Douglas-Rachford splitting, one has to calculate an iteration involving, iteratively, $\prox_{\alpha \|\cdot\|_S}$ and the range projector $Q_R$, and this would be a more accurate way of calculating minimizers of \eqref{eqin}. From this viewpoint, out method can be seen as performing only the first step of such a splitting. The main obstacle of performing more than one step is the application of the range projector, which is not easily available; it could be found by a costly QR decomposition of the Radon transform. Alternatively, one could try approximate range projectors using the range characterization via moments; this has been used, e.g., in \cite{Monard_2016}.
\end{remark}

In general, the benefit of the use of backprojection formulas is their computational efficiency. Note that advanced Tikhonov regularization using, e.g., total variation regularization, yields excellent results but is inferior to our proposed methods in terms of computational speed. We now discuss the computation of the proximal operator for the different choices of our proposed sinogram norms.

\subsection{Proximal maps for $\|\cdot\|_{R_t^{s,p}}^p$ and results}

First, we calculate the proximal operator for the Sharafutdinov norm $\| \cdot \|_{R_t^{s,p}}^p$. Note that it is more convenient to use the $p$-th power of the norm rather than the norm itself. Evaluating the proximal operator involves minimizing 
    \begin{equation*}  
        \frac{1}{2} \|g - y\|_{\R^{N}}^2 + \alpha \frac{1}{p} \|v_{s,t} \Fs g\|_{L^p}^p \,,  
    \end{equation*}
where as before $\Fs$ denotes the Fourier transform with respect to the $\sigma$-variable and $v_{s,t}$ is the weight given in \eqref{vweight}. Note that $|\xi|^2$ in this weight is the symbol of the Laplace operator. We thus implement the weight by using the 1-D Fourier transform of a simple second-order difference stencil $[-1 \ 2 \ -1]$ and evaluate $v_{s,t}$ accordingly. For the numerical computations, we also included an additional  parameter $\scal>0$ into $v_{s,t}$ which scales the frequency domain as follows:
    \begin{equation*}
        v_{s,t}(\xi) =  (\scal + |\xi|^2)^\frac{s}{2}
        \left(\frac{|\xi|}{(\scal + |\xi|^2)^\frac{1}{2} } \right)^t \,.
    \end{equation*}
The proximal operator can be calculated quite easily by observing that due to Parseval's relation, using $\hat{g} = \Fs g$ and $\hat{y} =\Fs y$, there holds
    \begin{align*} 
        &\frac{1}{2} \|g - y\|_{\R^{N}}^2 + \alpha \frac{1}{p}\|v_{s,t} \Fs g\|_{L^p}^p
        = \frac{1}{2} \|\hat{g} - \hat{y}\|_{\R^{N}}^2 + \alpha \frac{1}{p} \|v_{s,t}  \hat{g}\|_{L^p}^p \,.
    \end{align*}
Now, the minimizer of this functional is given by 
    \begin{equation}\label{Shaprox}
        \hat{g} = \prox_{\alpha |v_{s,t}|^p,l^p}(|\hat{y}|) \frac{\hat{y}}{|\hat{y}|} \,, 
    \end{equation}
where $\prox_{\alpha |v_{s,t}|^p,l^p}(|\hat{y}|)$ is the proximal $l^p$ operator on $\R$ with weight $\alpha |v_{s,t}|^p$, which can be applied component-wise on $|\hat{y}|$. Even though, except for certain cases of $p$ (e.g. $p = 1,\infty,2$), no explicit formula for the $l^p$-proximal map is available, it can be calculated efficiently via a Newton method.

\begin{figure}[ht!]
    \includegraphics[width=0.32\textwidth,trim=60 10 50 0, clip]{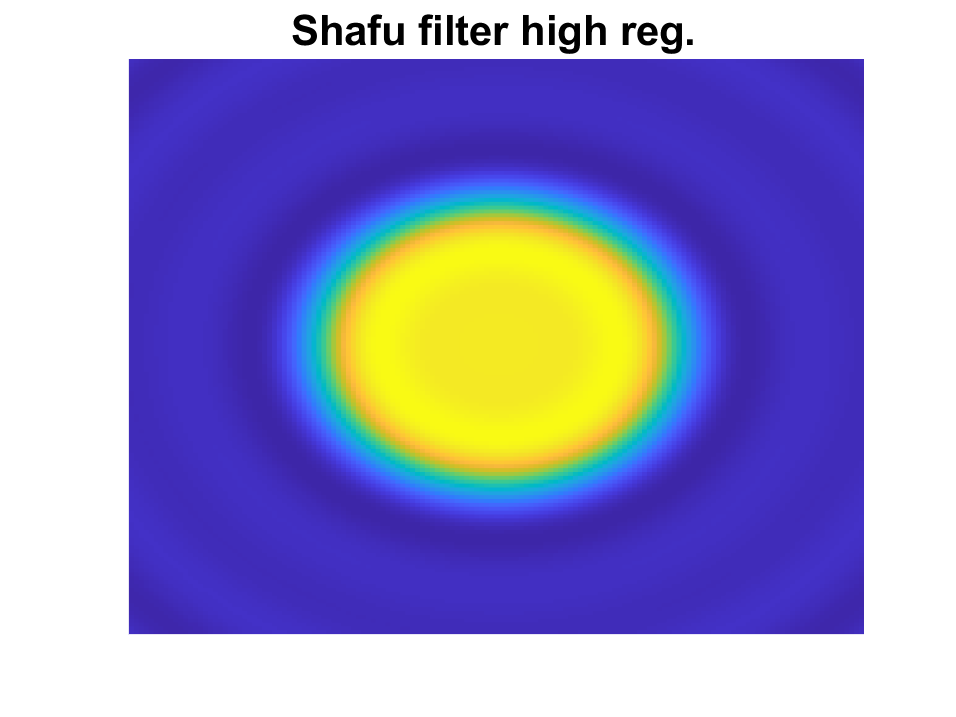}\hfill 
	\includegraphics[width=0.32\textwidth,trim=60 10 50 0, clip]{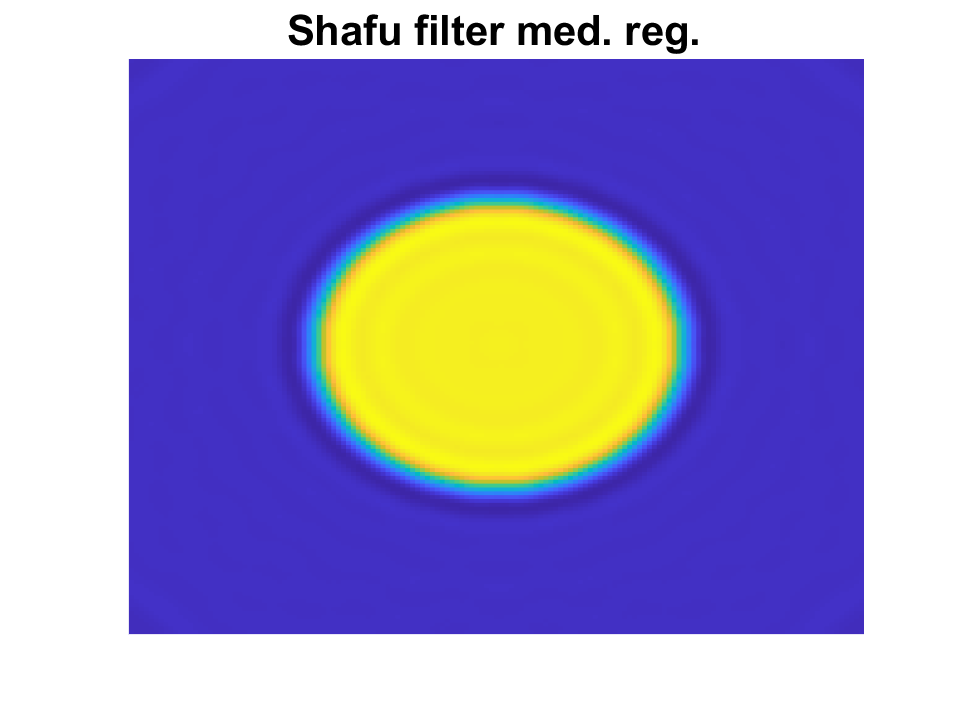}\hfill 
	\includegraphics[width=0.32\textwidth,trim=60 10 50 0, clip]{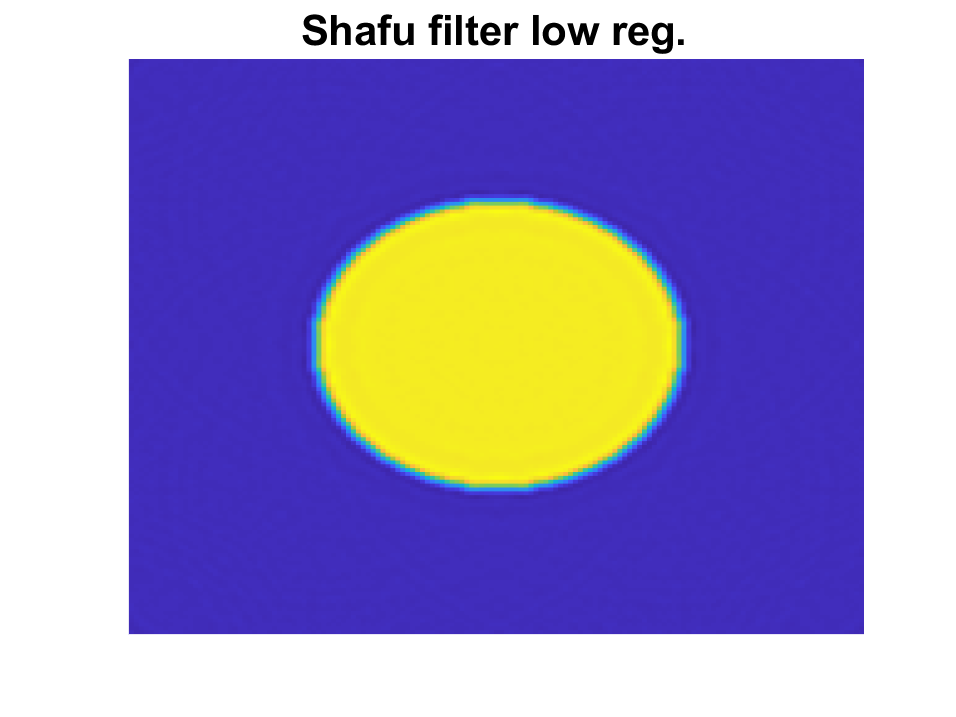}
	\includegraphics[width=0.32\textwidth,trim=60 10 50 0, clip]{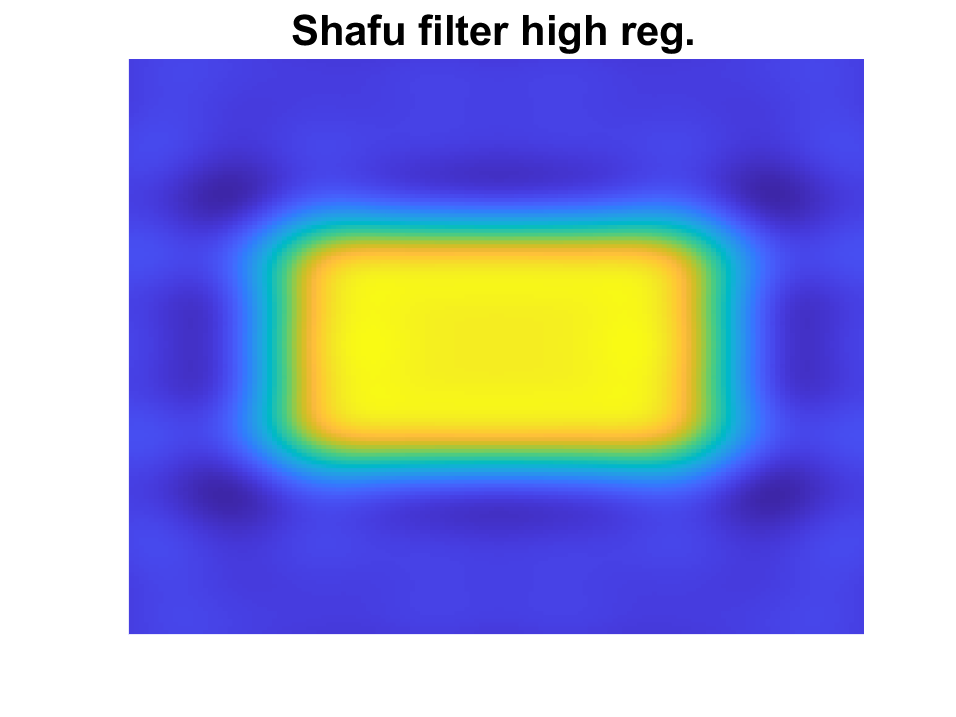}\hfill 
	\includegraphics[width=0.32\textwidth,trim=60 10 50 0, clip]{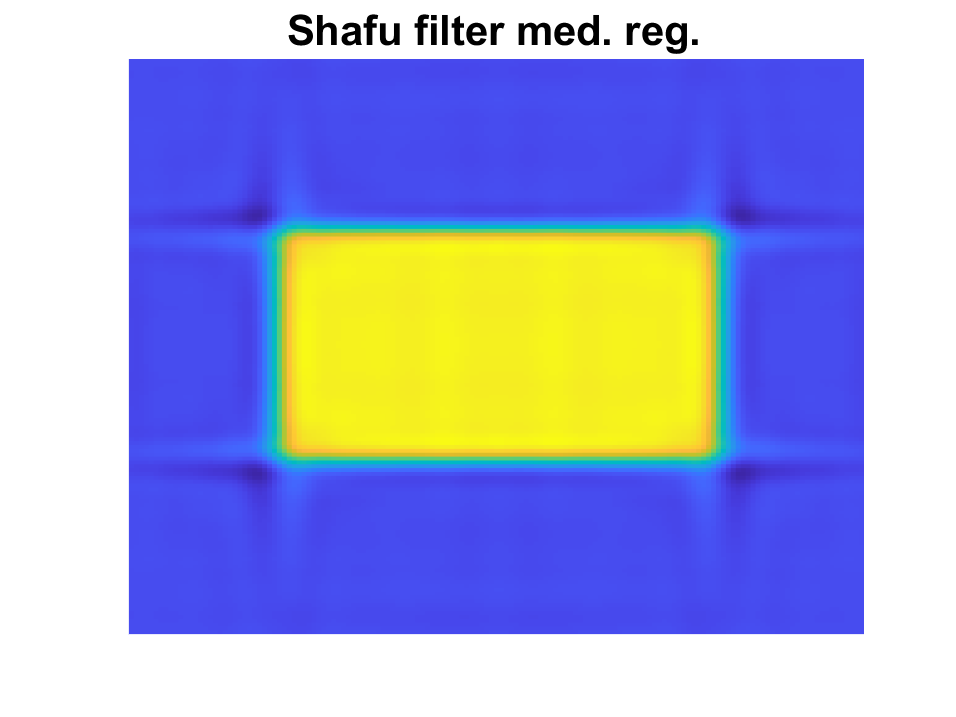}\hfill 
	\includegraphics[width=0.32\textwidth,trim=60 10 50 0, clip]{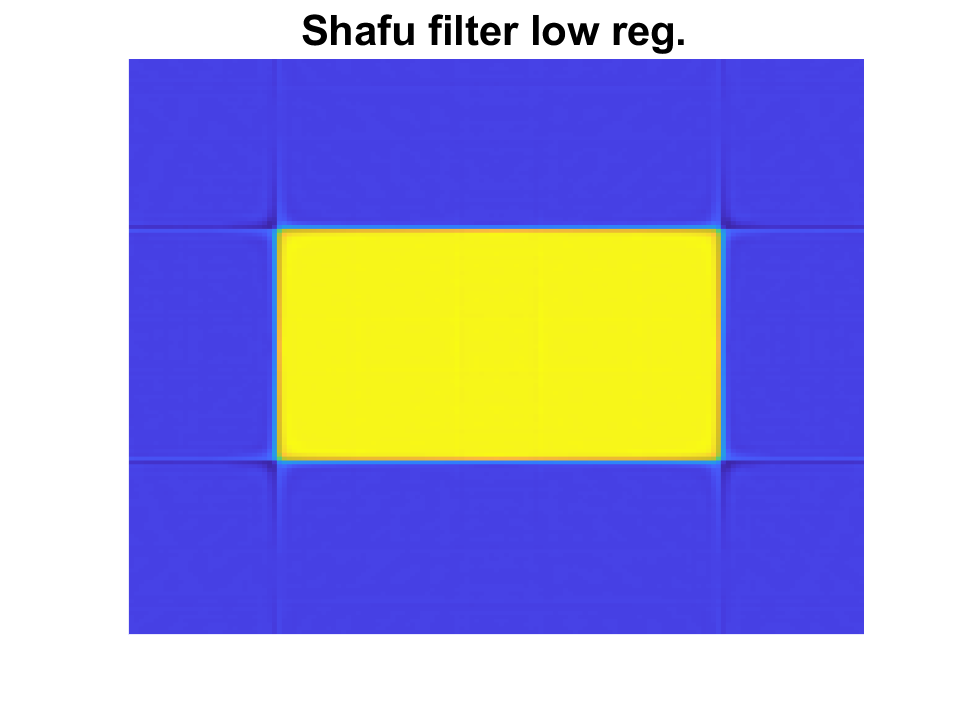} 
	\includegraphics[width=0.32\textwidth,trim=60 10 60 0, clip]{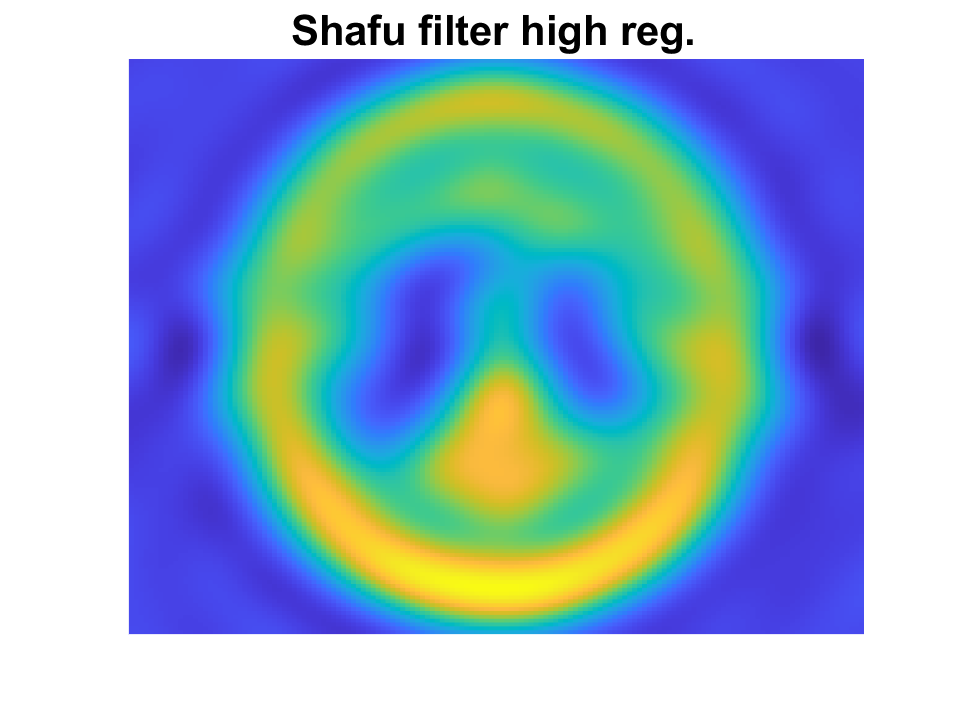}\hfill 
	\includegraphics[width=0.32\textwidth,trim=60 10 60 0, clip]{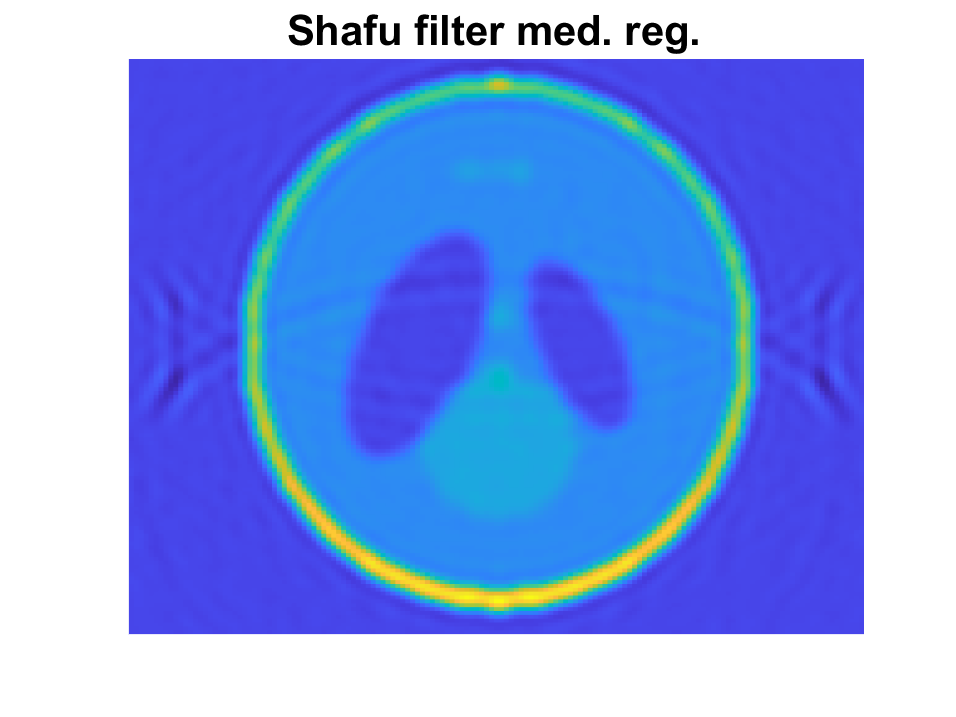}\hfill 
	\includegraphics[width=0.32\textwidth,trim=60 10 60 0, clip]{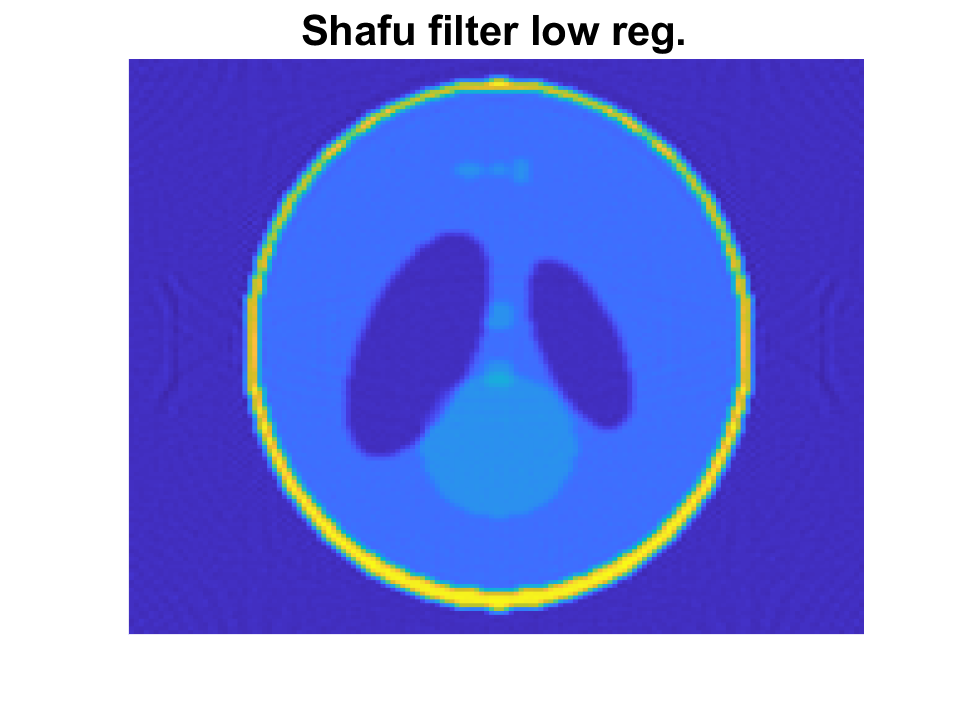}
	\caption{Comparison of the Sharafutdinov filter combined with  filtered backprojection (iradon) for various choices of the regularization parameter and for different exact solutions (ordered row-wise). Left column: high regularization, middle column: medium regularization, right column: low regularization.}
	\label{Fig2} 
\end{figure}

In Figure~\ref{Fig2}, we illustrate the method \eqref{irbb} for the Sharafutdinov filter/norm with exact data, and the parameters $p = 1.01$, $ s=1.1$, $t = 0.1$, $\scal = 1.\text{d-}12$, and for varying $\alpha$ based on Matlab's implementation of the Radon transform. As the ground truth, we use indicator functions of a circle and a rectangle, as well as Matlab's Shepp-Logan phantom, each of size 200 $\times$ 200 pixel. The leftmost column corresponds to high regularization and the rightmost to low regularization. The rows correspond to the different ground truths. The results are reasonable for small regularization. At the edges of the samples both artefacts as well as slight instances of the Gibbs phenomenon appear. Additionally, for the Shepp-Logan phantom, there are unclear artefacts for medium regularization.

\subsection{Proximal maps for $\no$ and results}

In this section, we consider $\|\cdot\|_S = \|\cdot\|_{\no}$, and sketch how to calculate $\prox_{\alpha \|\cdot\|_{\no}}$ numerically in a discrete setting. For this, we discretize the $\sigma$-derivative by a simple forward difference operator, denoted by $D_\sigma$, and the $L^2$-norm is replaced by a pixelwise Euclidean norm. We then have to minimize
    \begin{equation} \label{proxmin} 
        \frac{1}{2} \|g - y\|_{\R^{N}}^2 + \alpha \|D_\sigma g\|_{L^\infty\text{-}L^1}\,.  
    \end{equation}
In the above functional, with a slight abuse of notation, we consider $g$ and $y$ as vectors in the first term, while $g$ is identified with a matrix (arranged by offset $\times$ angle indices) in the second expression. Hence, $D_\sigma g = h$ can be represented by a matrix $h = h_{i,j}$, and the ${L^\infty\text{-}L^1}$-norm is then given by 
    \begin{equation*}
         \|h\|_{L^\infty\text{-}L^1} = \sup_{j} \sum_{i=1}^{N} |h_{i,j}| \,.
    \end{equation*}
The structure of problem \eqref{proxmin} allows the application of well-known efficient convex optimization tools such as FISTA \cite{BeTe}, the Chambolle-Pock method \cite{ChPo}, or an augmented Lagrangian (aka split Bregman) method \cite{split}. For our examples, following tests of some of these approaches, we decided on using the continuation penalty method of \cite[Alg.~2]{WangYang}. A central ingredient in all of these methods is the calculation of the proximal operator for the gradient norm, in our case the ${L^\infty\text{-}L^1}$-norm. While there is no explicit formula in this setting, an efficient numerical algorithm can for example be found in \cite{Bejar}. 

Note that although \eqref{proxmin} looks quite similar to a standard total-variation denoising problem (the ROF-filter), it has the significant difference that the total variation involved is essentially one-dimensional, since only the $\sigma$-gradient is taken. It makes quite a difference for total-variation filtering if it is done in one or two dimension. For the 1D-case, there are some interesting algorithms available, for example the taut-string method \cite[Sec. 4.4]{Scherzbook}, \cite{Taut1,Taut2}; another interesting property of  1D-total variation not valid in the 2D case is found in \cite[Section~4.2]{KiWo}. Leveraging the possibilities of one-dimensional total variation filtering for computing the proximal operator is relegated to future work.

\begin{figure}[ht!]
	\includegraphics[width=0.32\textwidth,trim=60 10 50 0, clip]{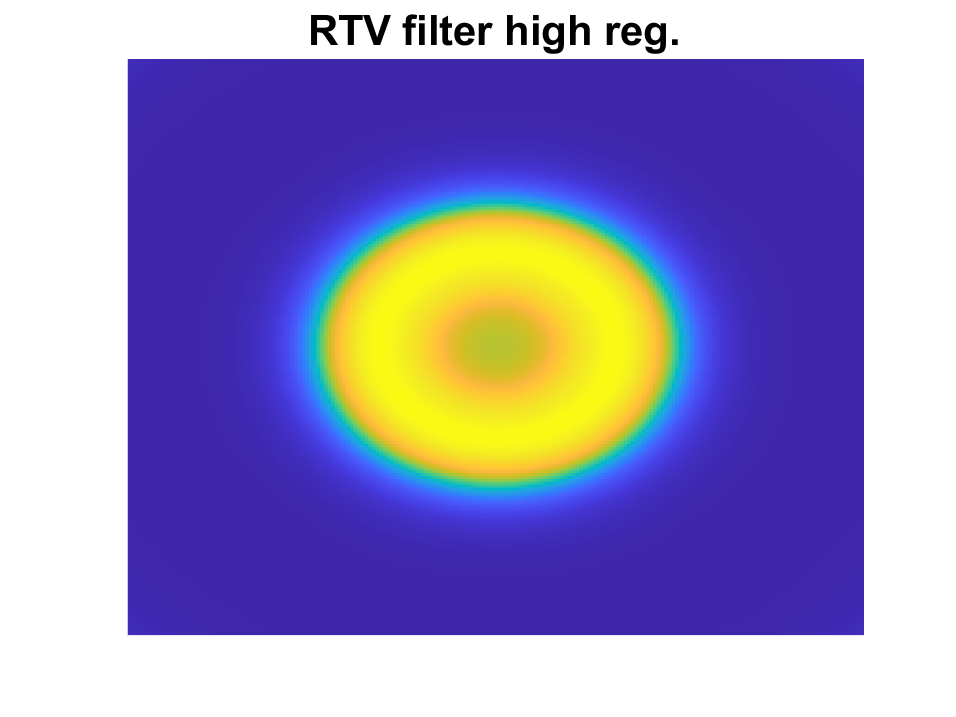}\hfill 
	\includegraphics[width=0.32\textwidth,trim=60 10 50 0, clip]{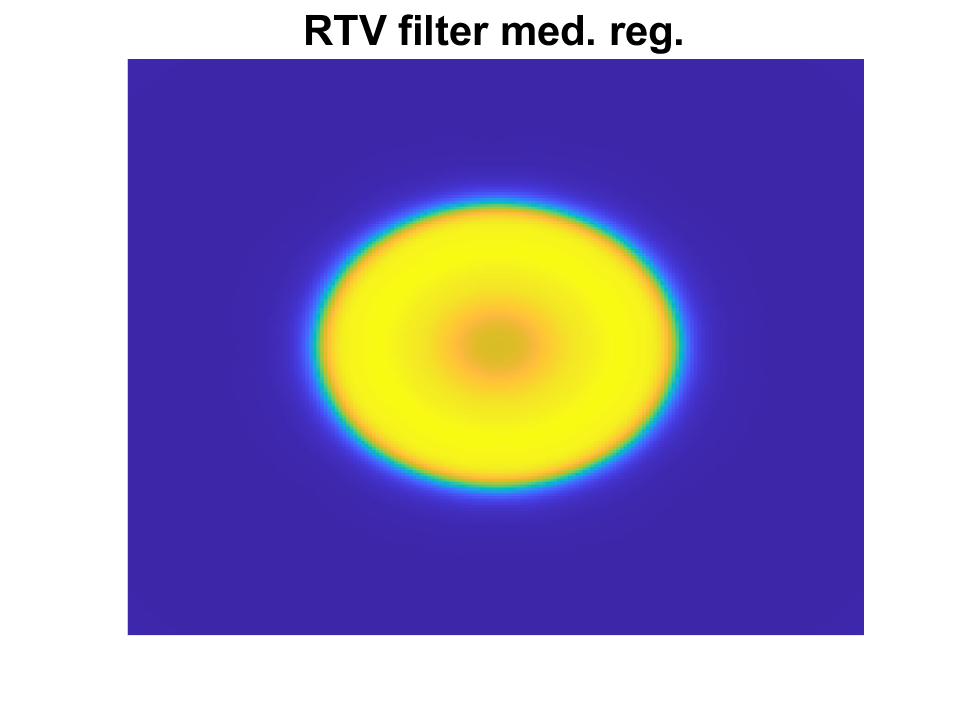}\hfill 
	\includegraphics[width=0.32\textwidth,,trim=60 10 50 0, clip]{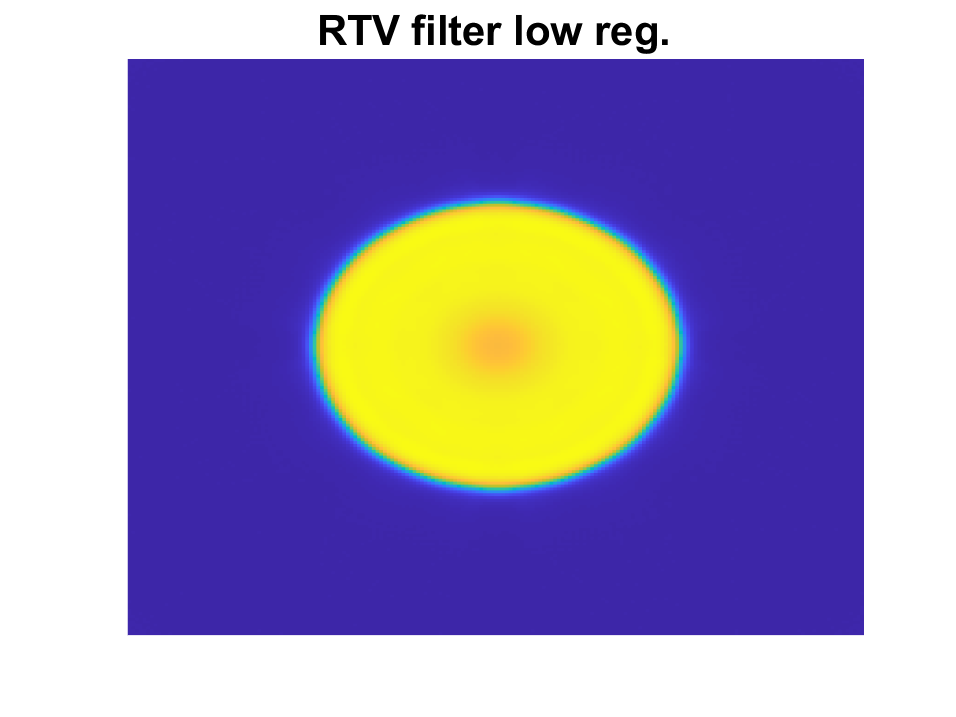}
	\includegraphics[width=0.32\textwidth,trim=60 10 50 0, clip]{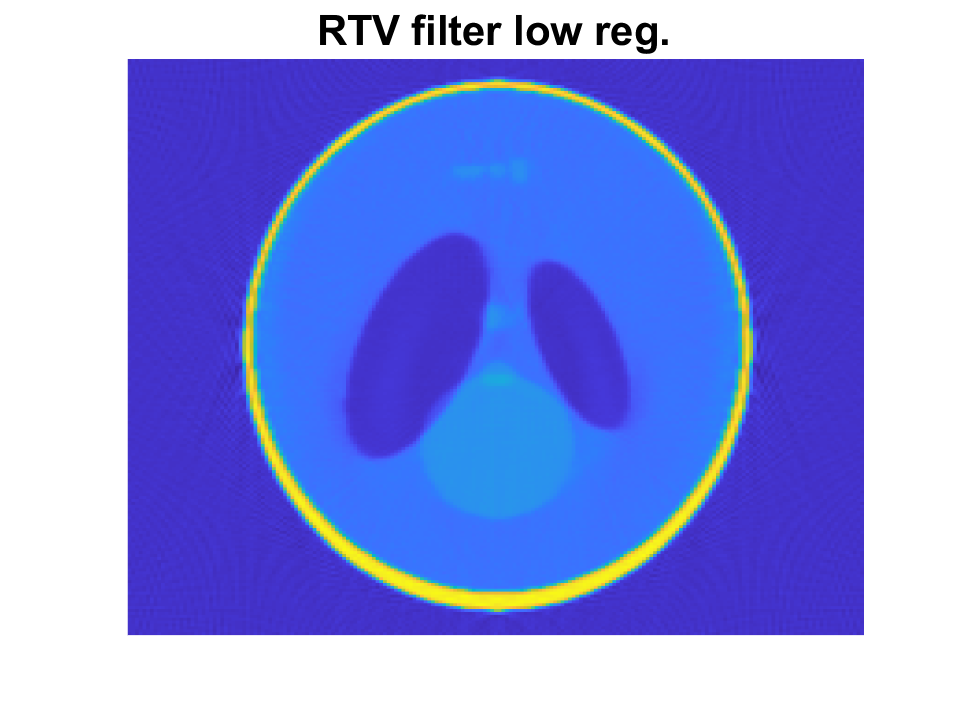}\hfill 
	\includegraphics[width=0.32\textwidth,trim=60 10 50 0, clip]{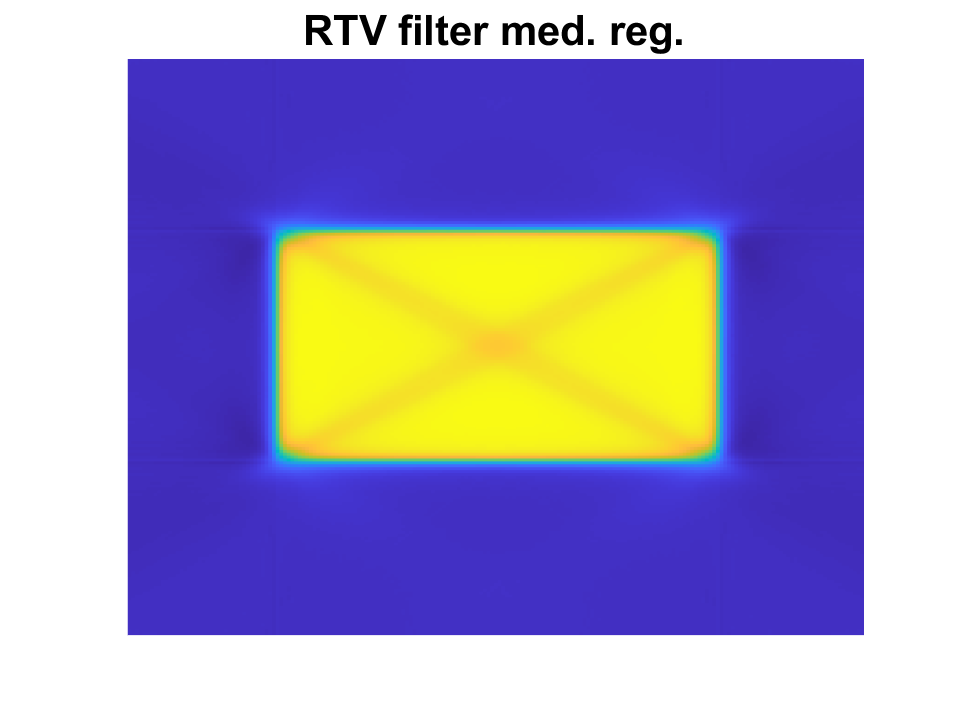}\hfill 
	\includegraphics[width=0.32\textwidth,trim=60 10 50 0, clip]{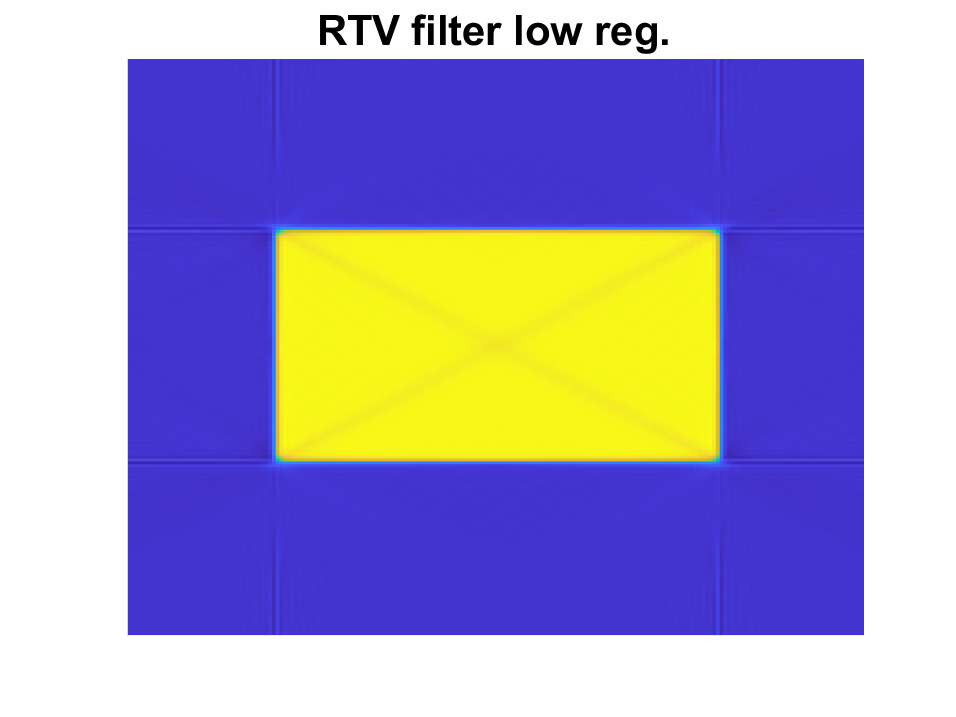} 
	\includegraphics[width=0.32\textwidth,trim=60 10 60 0, clip]{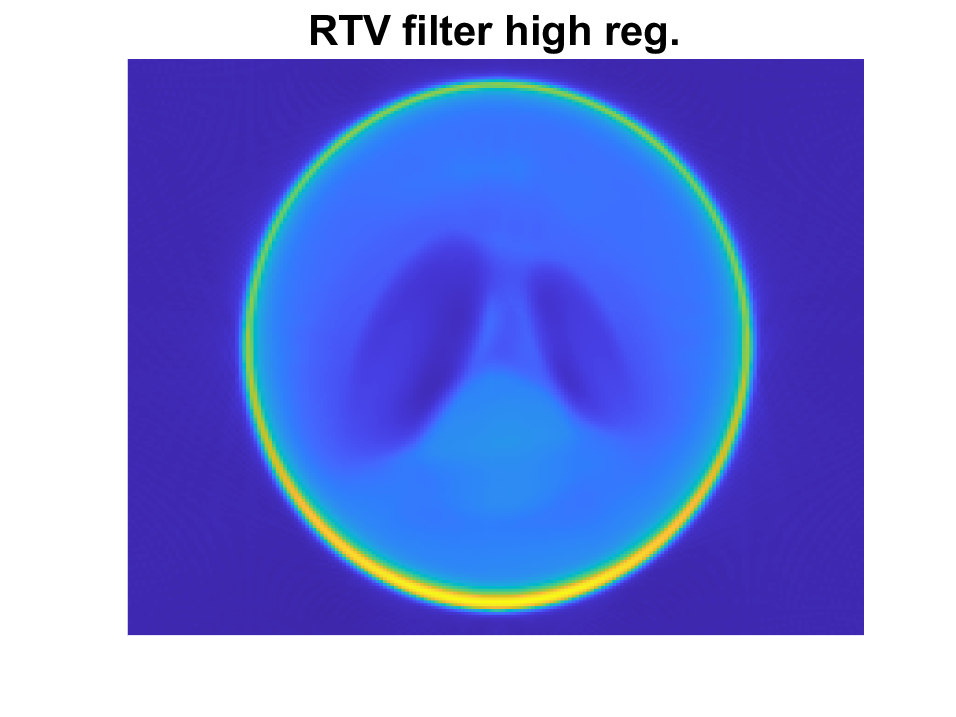}\hfill 
	\includegraphics[width=0.32\textwidth,trim=60 10 60 0, clip]{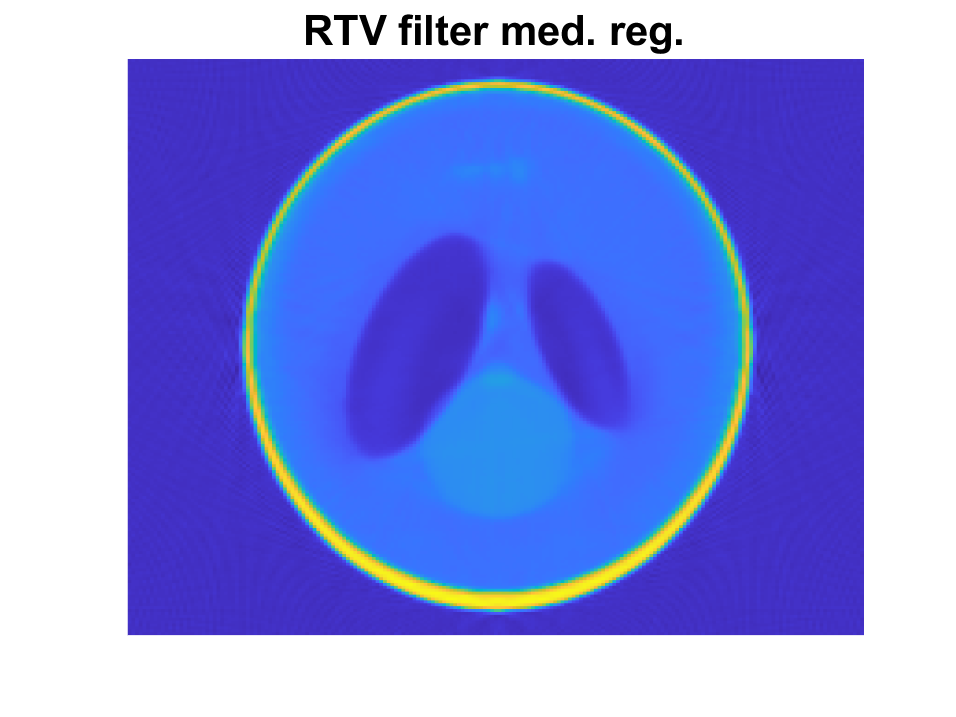}\hfill 
	\includegraphics[width=0.32\textwidth,trim=60 10 60 0, clip]{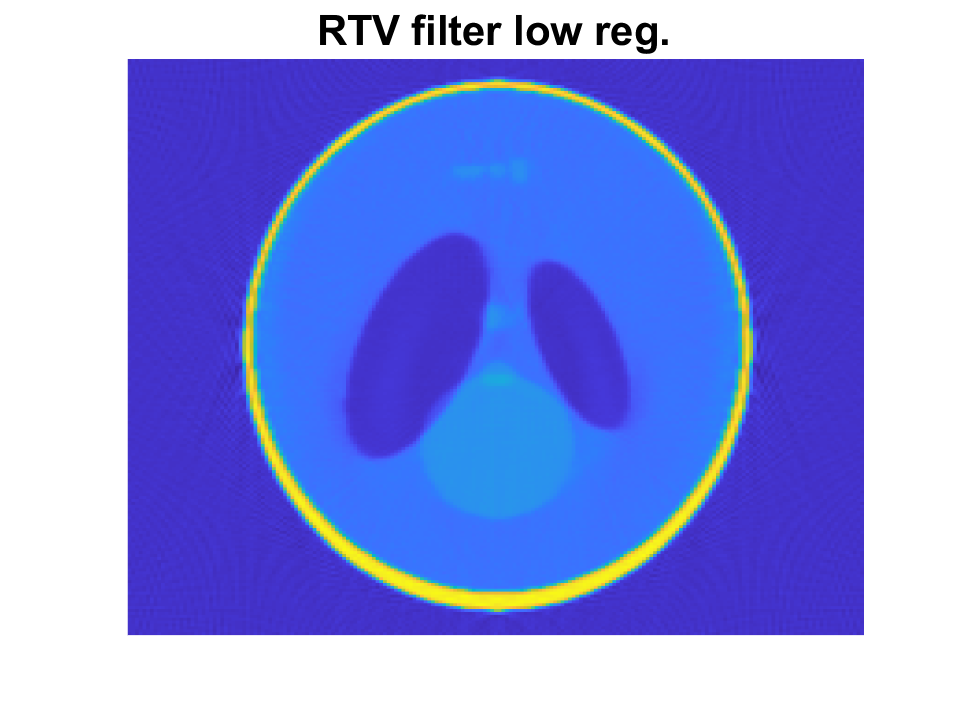}
	\caption{Comparison of the RTV filter combined with filtered backprojection (iradon) for various choices of the regularization parameter $\alpha$ and for different exact solutions (ordered row-wise). Left column: high regularization, middle column: medium regularization, right column: low regularization.}
	\label{Fig1} 
\end{figure}
 
Figure~\ref{Fig1} presents the results of using \eqref{irbb} with the RTV-filter. What can be observed from the circular and rectangular samples is that the RTV-filter allows for sharp edges, but removes certain parts which correspond to maxima of the Radon transform. For the circle image, this means that part of the center is ``removed'',  while for the rectangle, the diagonals are partially removed, since they are the cords of maximal length. We point out that the maximal height of the reconstructed solution (not visible in the plots) is quite close to the true height of the ground truth, i.e., we observe only a \emph{low loss of contrast}.

The behaviour of the artefacts and the low loss of contrast can actually be explained for the circle image via the following example: In the case of the image $f$ being the indicator function of a circle, as in \eqref{radsqu}, it is possible to calculate the RTV filter analytically under the hypothesis that $y_\alpha:= \prox_{\alpha \|\cdot\|_{RTV}}(R f)$ is radially symmetric and has convex level sets. Let $Rf$ be given by \eqref{radsqu}. Under the given assumptions, the filtered data satisfy  $\|y\|_{RTV} = 2 \|y\|_{\infty}$ due to Lemma~\ref{help1}. It then follows that $ \prox_{\alpha \|\cdot\|_{RTV}}(R f) =  \prox_{2\alpha \|\cdot\|_{\infty}}(R f)$, and the $L^\infty$-prox can be easily calculated (cf.~\cite[Ex. 6.48]{Beckbook}) using that $Rf\geq 0$: 
    \begin{equation*}
        \prox_{2\alpha \|\cdot\|_{\infty}}(R f)(\sigma,\theta) = \min\{Rf(\sigma,\theta),\beta\}\,,
    \end{equation*}
where $\beta\geq 0$ is a cutoff value which can be calculated from $\lambda$ ($\beta$ may become $0$ if the input has small enough $L^1$-norm). Thus, this filter simply cuts off the larges values of $Rf$. The difference between $Rf$ and $\prox_{2\alpha \|\cdot\|_{\infty}}(R f)$ is in this case even in the range of the Radon transform (and is radially symmetric), and its inverse can be calculated explicitly using integral formulas \cite[p.22]{Natterer}. 

\begin{figure}[ht!]
    \includegraphics[width=0.45\textwidth]{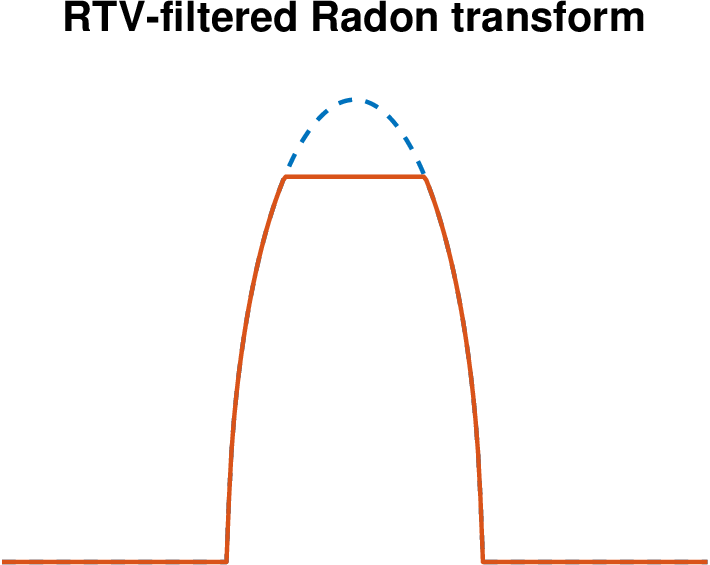} \hfill 
    \includegraphics[width=0.45\textwidth]{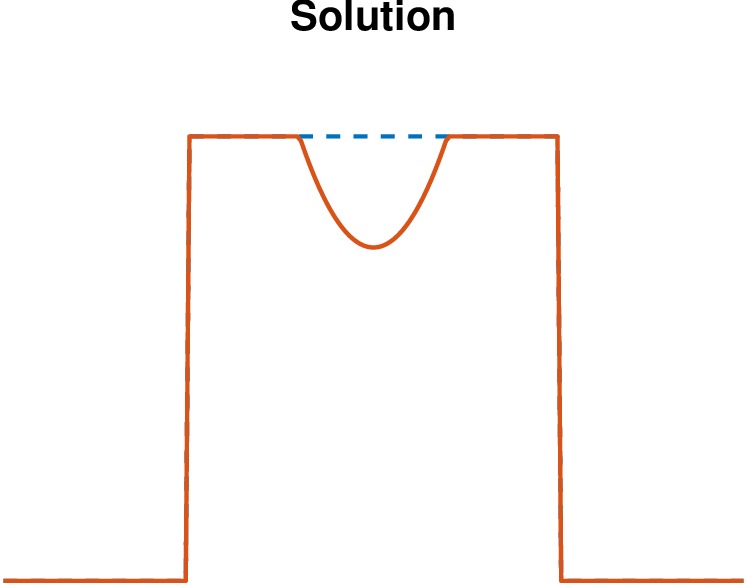}
    \caption{Illustration of the effect of an RTV-filter if $f$ is the indicator function of a disk. Left: Radon transform (dashed blue) and filtered Radon transform (red). Right: corresponding solutions in image space after applying $R^{-1}$.}\label{Fig0}
\end{figure}

These results are illustrated in Figure~\ref{Fig0}. On the left, we display the filtered data, $\prox_{2\alpha \|\cdot\|_{\infty}}(R f)$, in red, and, by blue dashed lines, the true data $Rf$ for some fixed angle. On the right, we display a cross section of the image resulting by applying the inverse Radon transform to the filtered data. It can be observed that edges are preserved, but parts of the mass in the center is removed. This analytical result perfectly matches those of the numerical experiments presented in Figure~\ref{Fig1}. Note that the fact that in this case the proximal mapping for $\no$ equals the $L^\infty$-proximal operator is not true in general, and does not hold in case of noisy data. Note also that in Figure~\ref{Fig1}, we do not have a loss of contrast apart from the removal of the center part: the filtered solution reaches the same height as the ground truth. This is different to a  \TV-regularization, where the height of the resulting solution is usually below the true height. This appears to be one advantage of our proposed method.


\begin{remark} \rm
The proposed sinogram norm filtering is a natural generalization of commonly used filtered backprojection formulas. The well-known inversion formula in Fourier space (here stated in 2D): \cite[Eq.~(2.4)]{Natterer}, 
    \begin{equation}\label{thisis}
        f(x) = \frac{1}{2 (2\pi)} R^* K_\sigma R f \,,
        \qquad \text{where} \qquad
        K_\sigma g := \F_1^{-1} |\sigma| \F_1 g 
    \end{equation}
leads to filtered backprojection when 
the multiplier $|\sigma|$ is replaced by an appropriate filter $K_\phi= \F_1^{-1} \phi(\sigma) \F_1$.
    
Suppose we would like to use Tikhonov regularization with, say, a $\|f\|_{H_0^s}$-norm as penalty, and by  following our rationale, we replace this penalty by the equivalent sinogram norm $\|R f\|_{H^{s + 1/2}}$, 
and for reconstruction use the operation $R^{-1}\prox_{\alpha \|\cdot\|_S}(y)$ as in \eqref{appfil}. It turns out that the the resulting regularization is exactly the above backprojection formula, where $\phi(\sigma) = |\sigma| w(\sigma)$ and $w$ the Fourier weight in the definition of  $\|\cdot\|_S$. Conversely, we can also obtain from suitable filtered backprojection formulas its interpretation as sinogram-norm regularization using a appropriate weights in  Fourier space. 

In this context, our proposed approach using the Sharafutdinov norm is a natural generalization of filtered backprojections.
Since the associated proximal map is based on the Fourier transform, we can implement the Sharafutdinov filtering using \eqref{bappfil} in the form of the following nonlinear backprojection formula:
    \begin{equation*}
        f(x) \approx  \frac{1}{2 (2\pi)} R^* K_{|\sigma| \prox_{S}(y)} y \,,
    \end{equation*}
where $\prox_{S}$ is given by \eqref{Shaprox} (and which depends nonlinearly on the data). A conceptually similar nonlinear filtering has been proposed recently by using frames \cite{Halti2024}; cf.~also \cite{HuRa2021}. 
\end{remark}

\subsection{Noisy and incomplete data}

We also tested our proposed method with noisy data, the results of which can compete with any usual regularization method. Due to limited space, we do not report the details here, but instead refer to the preprint \cite{KiHuarxiv}. For the Sharafutdinov filter, we observe the highest senstivity with respect to the parameters $s$ and $p$, while $t$ plays a negligible role. The computation time for the proposed method was around 5--10 times higher than standard backprojection. 

We now consider the incomplete data case, for which we show that our sinogram-space filtering approach has a striking advantage compared to Tik\-ho\-nov-like regularization methods. For this, consider incomplete sinogram data which are obtained by applying a filtering matrix $W$ to the data, where $W$ is a projection matrix which returns only part of the sinogram data. The model equation is then $W R f = y_{\text{inc}}$, where $y_{\text{inc}}$ represents the incomplete data. Such a model can be easily integrated into our filtering approach, since we only have to replace the usual prox-operator by
    \begin{equation*}
        \prox_{\text{inc}}(y_{\text{inc}}) = \argmin_z \frac{1}{2} 
        \|W z -  y_{\text{inc}}\|^2 + \alpha \|z\|_S \,,
    \end{equation*}
and, as before, apply an approximate inversion yielding $f = R_\alpha^\# \prox_{\text{inc}}(y_{\text{inc}})$. In contrast to other regularization methods, e.g., \cite{EHN,Ho,SKHK}, we can use simple inversion methods like iradon, and do not have to set up a forward model operator. (Even for modern advanced iterative method, one typically requires an implementation of $R$ \emph{and} its adjoint, which can be nontrivial).

\begin{figure}[ht!]
\begin{minipage}{0.32\textwidth}
\begin{center}
{RTV}\\[1ex]
\includegraphics[width=\textwidth,trim=50 10 50 0, clip]{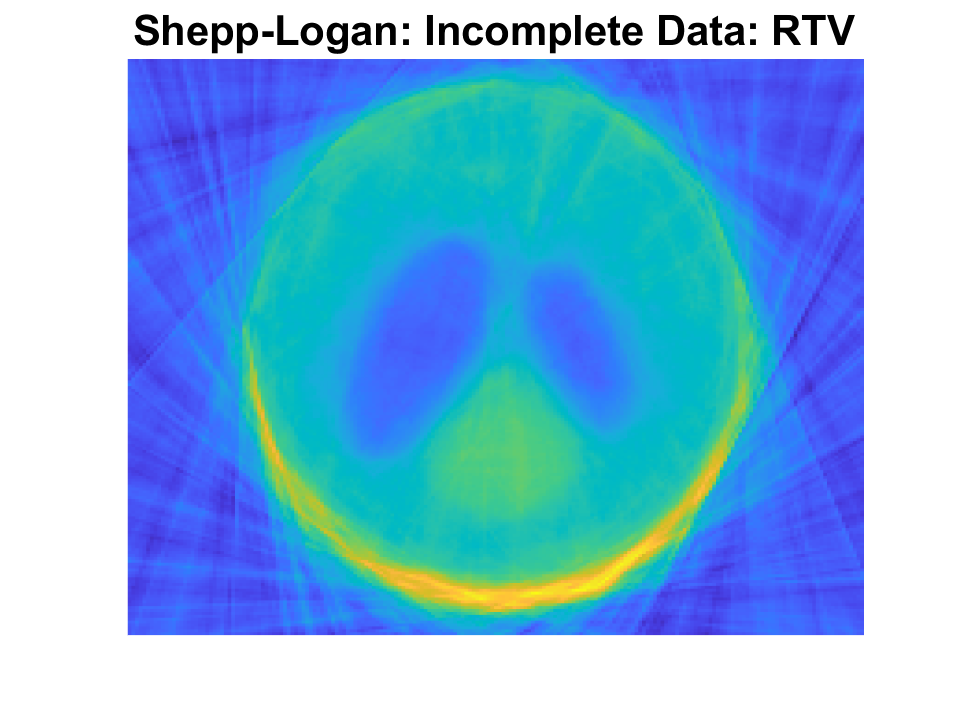}\\
 \includegraphics[width=\textwidth,trim=50 10 50 0, clip]{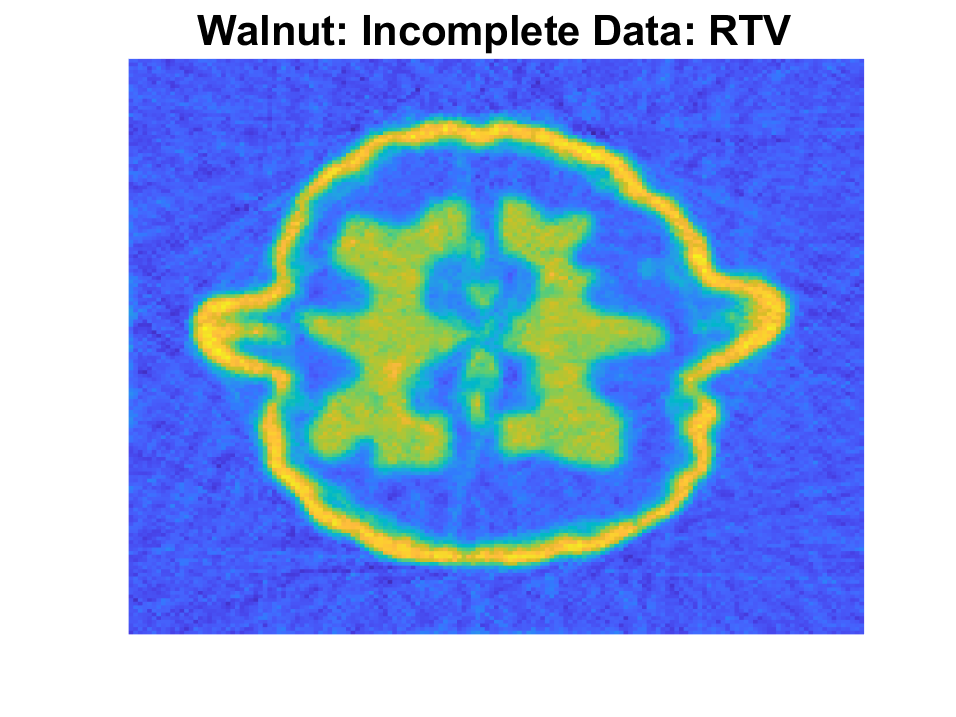}\\
 \includegraphics[width=\textwidth,trim=50 10 50 0, clip]{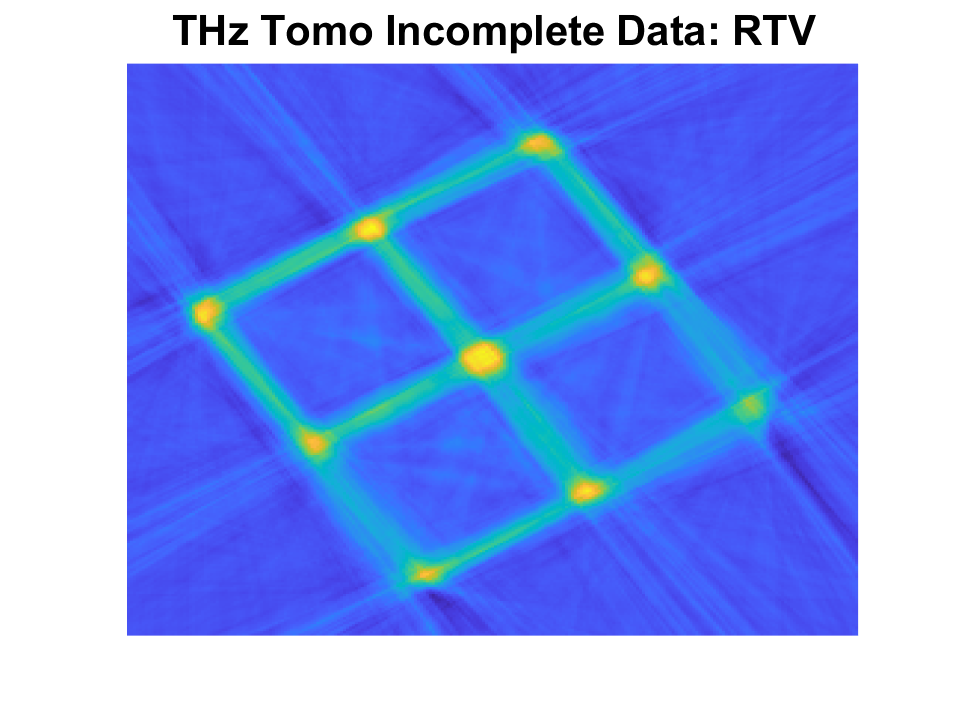}
\end{center}
\end{minipage}  
\begin{minipage}{0.32\textwidth}
\begin{center}
{Sharafutdinov}\\[1ex]
\includegraphics[width=\textwidth,trim=50 10 50 0, clip]{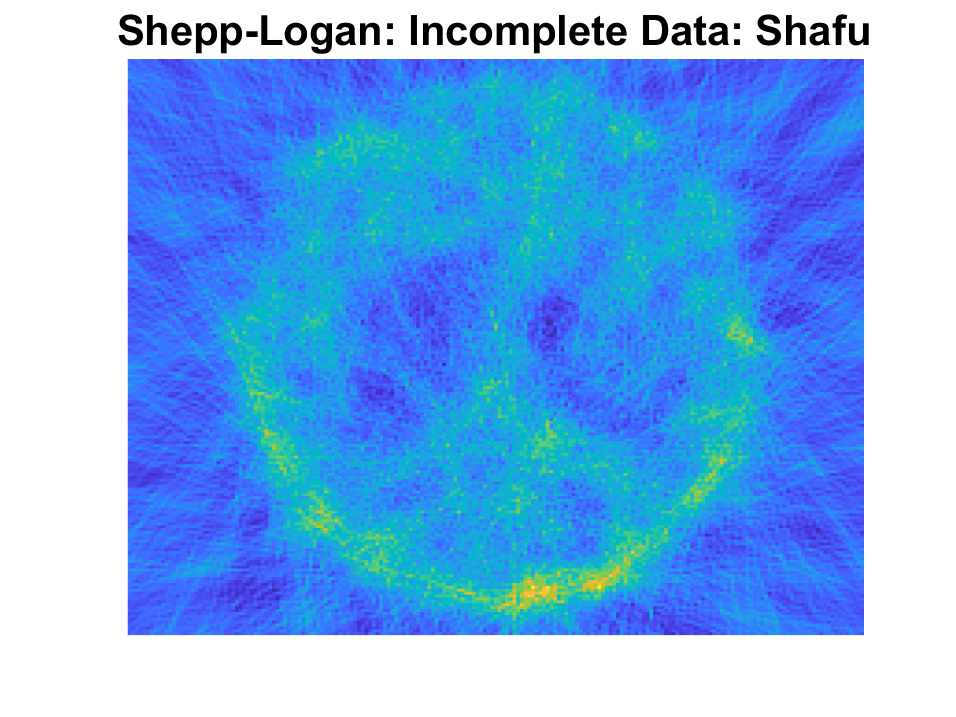}\\
 \includegraphics[width=\textwidth,trim=50 10 50 0, clip]{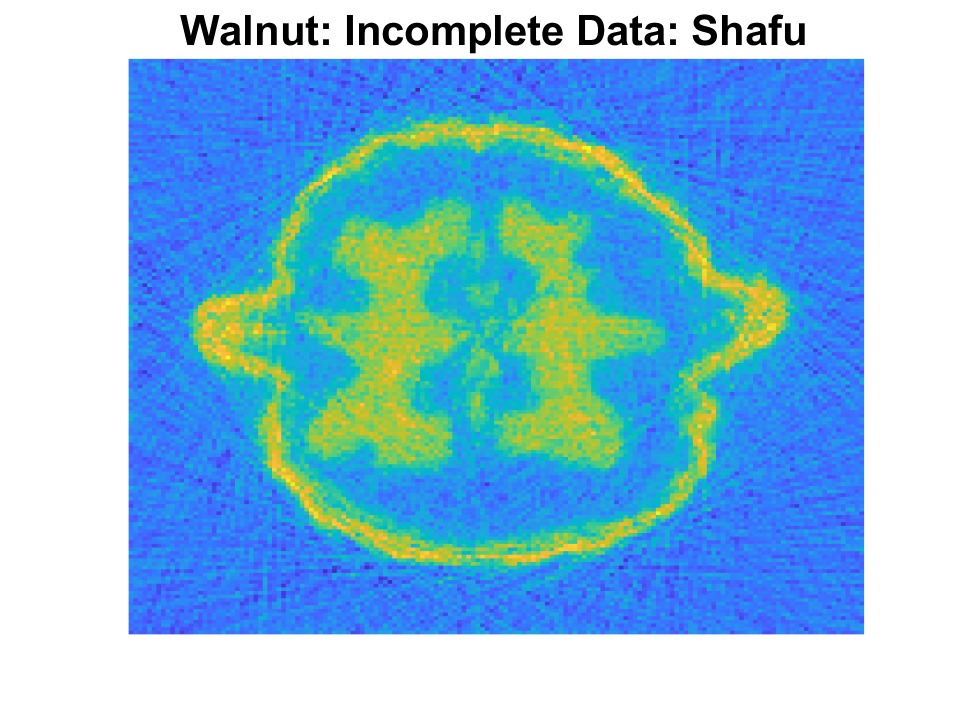}\\
 \includegraphics[width=\textwidth,trim=50 10 50 0, clip]{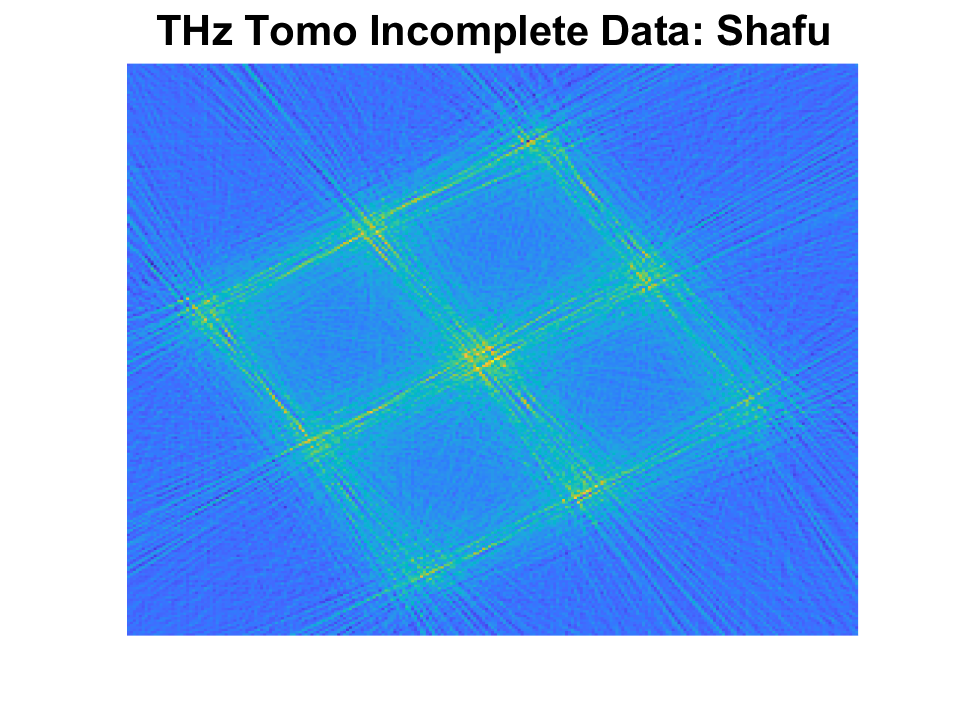}
\end{center}
\end{minipage}  
\begin{minipage}{0.32\textwidth}
\begin{center}
{CGNE}\\[1ex]
\includegraphics[width=\textwidth,trim=50 10 50 0, clip]{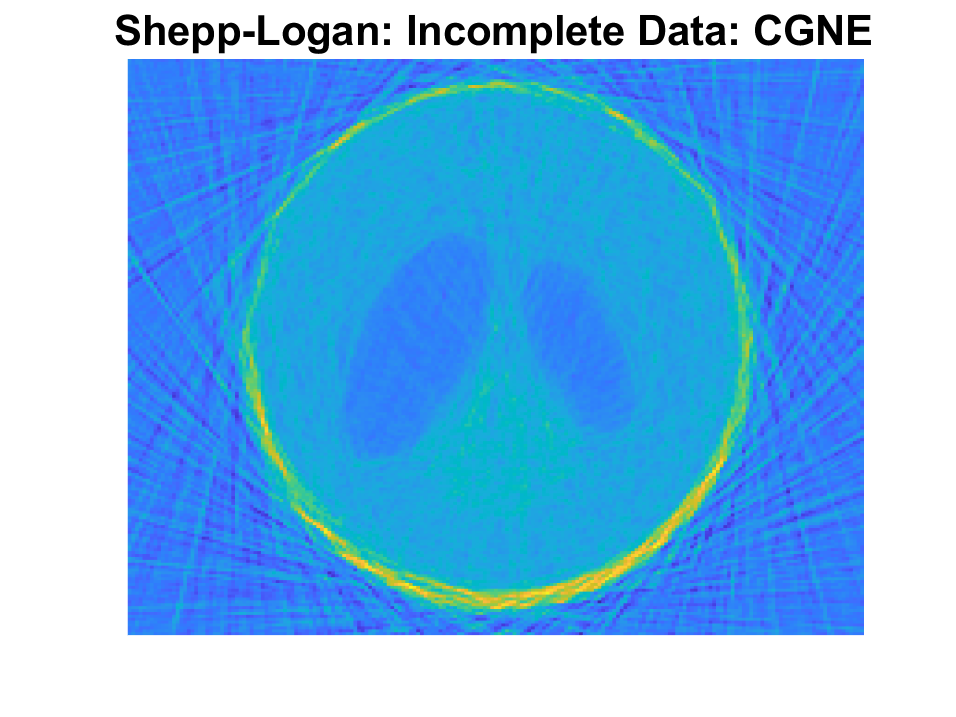}\\
 \includegraphics[width=\textwidth,trim=50 10 50 0, clip]{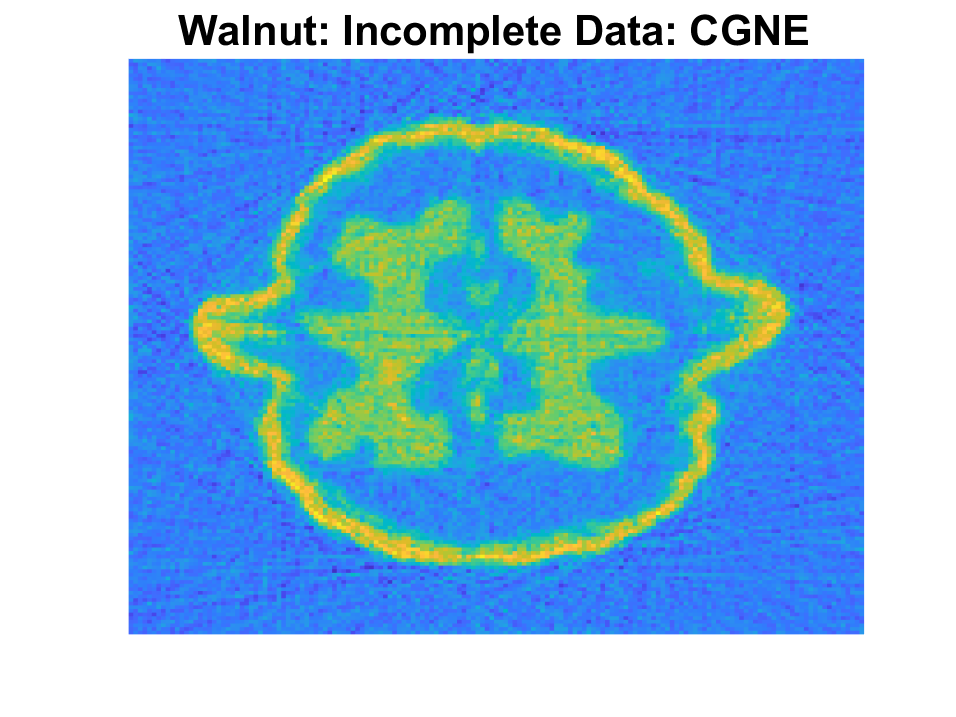}\\
 \includegraphics[width=\textwidth,trim=50 10 50 0, clip]{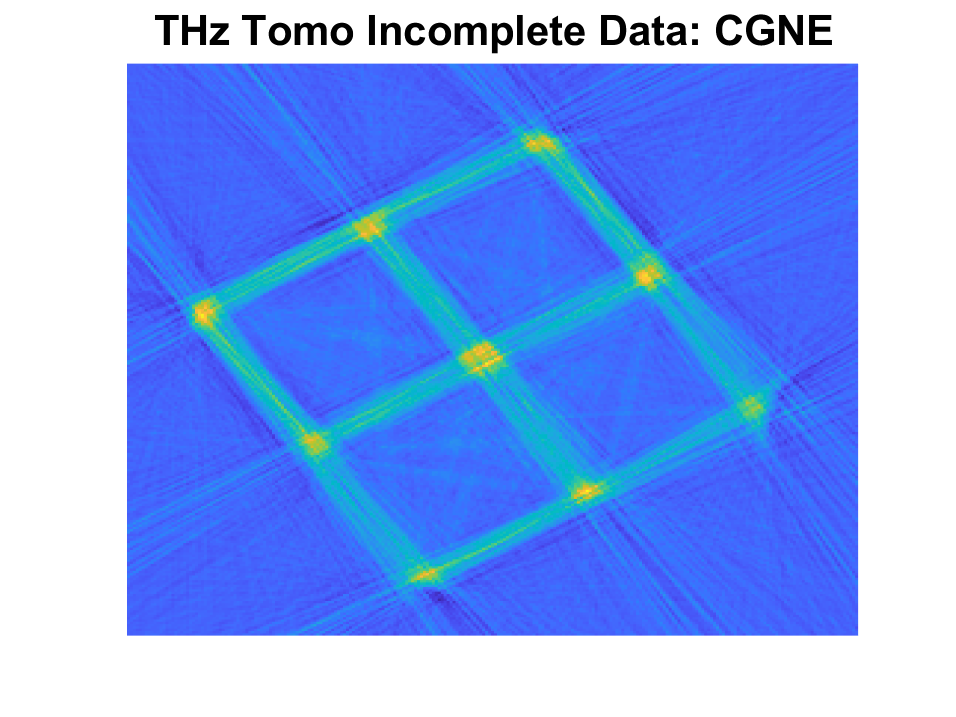}
\end{center}
\end{minipage}
\caption{Reconstruction using incomplete data: Columns: methods RTV, Sharafutdinov-filter, and CGNE inversion. Rows: Shepp-Logan phantom from Matlab, Walnut data from \cite{fips,Finnish}, and THz tomography data from \cite{Foso22,Hub22}.
}
\label{figinc}
\end{figure}

We tested this approach for (i) the usual Matlab Shepp-Logan model, (ii) for the Walnut X-ray data from the Finnish Inverse Problems Society \cite{fips,Finnish}, and (iii) for data from  terahertz (THz) tomography for nondestructive testing from \cite{Foso22,Hub22}. In case (i), we use 10\% percent of the data for a $200 \times 200$ pixel ground truth, in case (ii) 50\% data for the $164\times 164$ pixel ground truth dataset (Data164.mat), and in case (iii) we use 25\% of (downsampled) data for a $264 \times 264$ pixel ground truth. The data were selected by randomly picking the corresponding amount of indices in both the angular and offset variable. For comparison, we also consider the corresponding result for incomplete data using a CGNE iterative reconstruction; see~\cite{EHN,HaCGNE}. All parameters were selected to obtain minimal errors. For the Sharafutdinov filter, the relevant parameters were set to $p\sim 1.01$ and $s \sim 1$. The results are illustrated in Figure~\ref{figinc}. We observe a visually appealing result for the RTV-filter, while the result with the Sharafutdinov filter is less appealing. Compared to CGNE, the RTV-filter has fewer fine-structured artefacts and a better contrast. As mentioned above, CGNE requires an implementation of the forward method and its adjoints (we used a sparse matrix representation), which restricts CGNE in the THz case to downsampled data, while the sinogram filters were free of such a constraint.

\begin{figure}[ht!]
\includegraphics[width=0.31\textwidth,trim=50 10 50 0, clip]{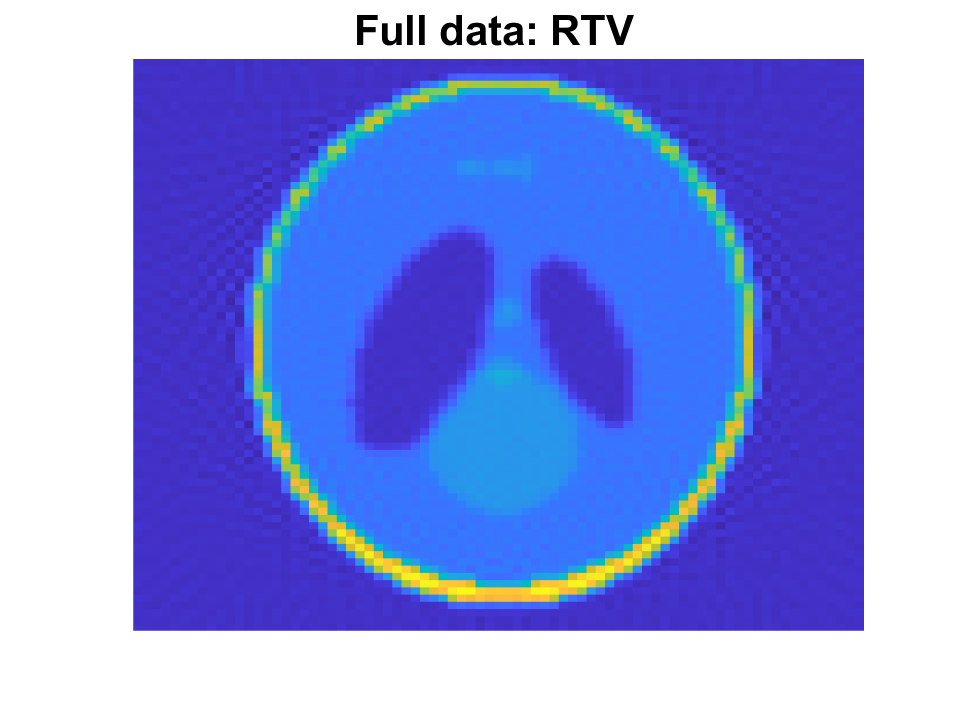} 
\includegraphics[width=0.31\textwidth,trim=50 10 50 0, clip]{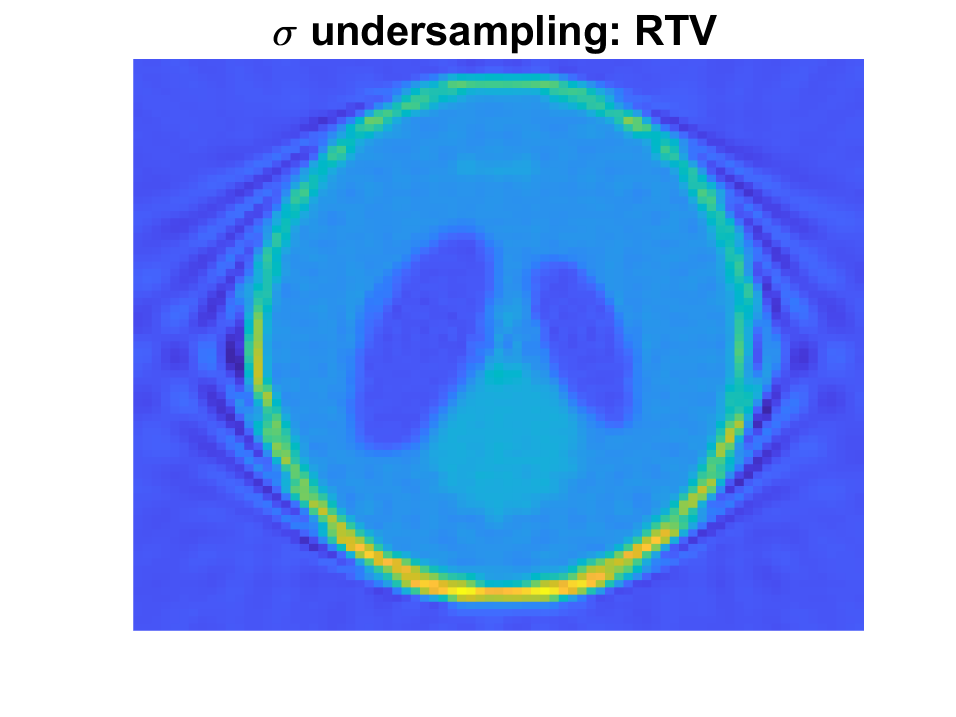} 
\includegraphics[width=0.31\textwidth,trim=50 10 50 0, clip]{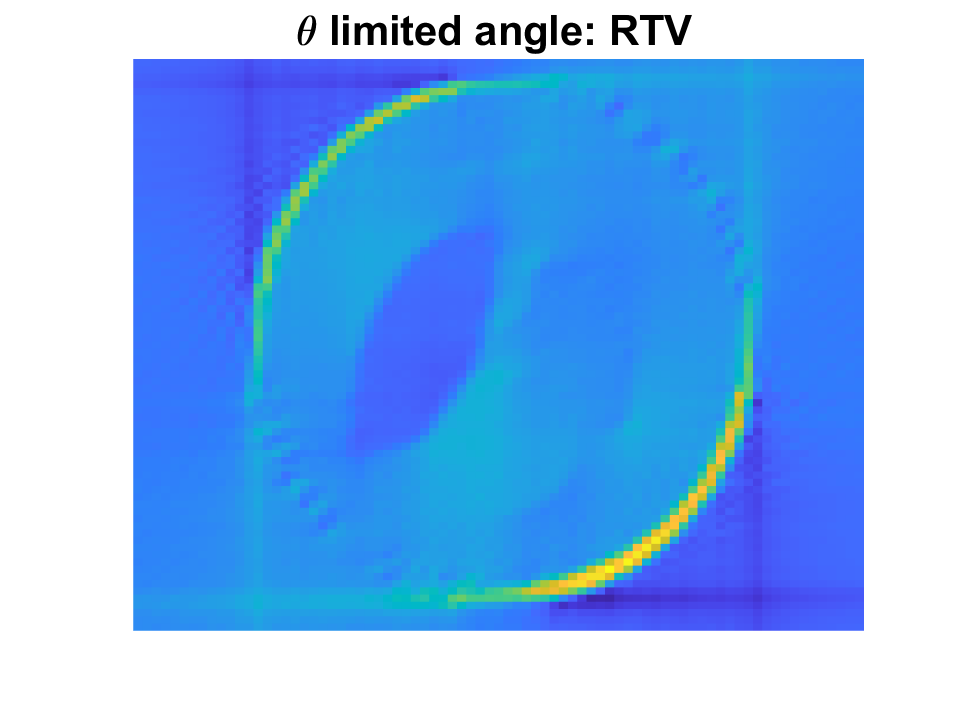} \\
\includegraphics[width=0.31\textwidth,trim=50 10 50 0, clip]{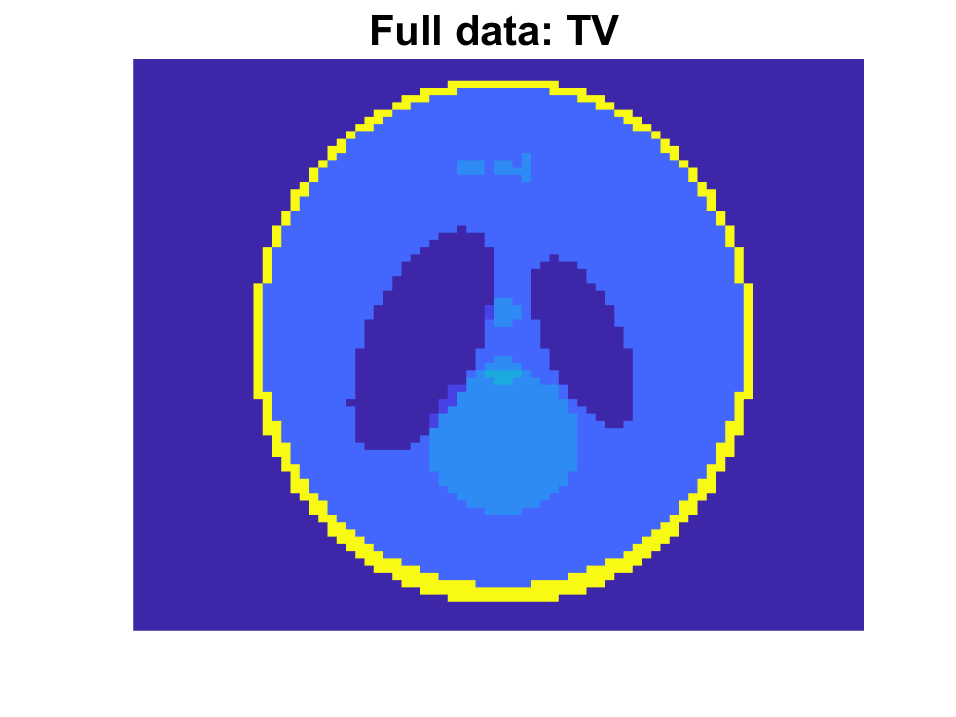} 
\includegraphics[width=0.31\textwidth,trim=50 10 50 0, clip]{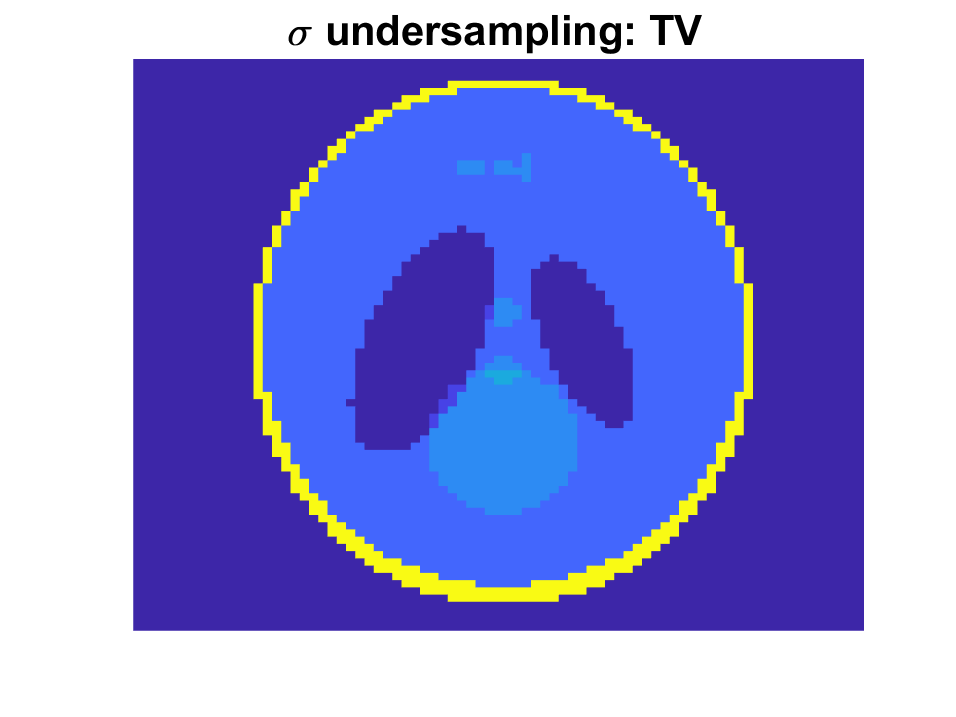} 
\includegraphics[width=0.31\textwidth,trim=50 10 50 0, clip]{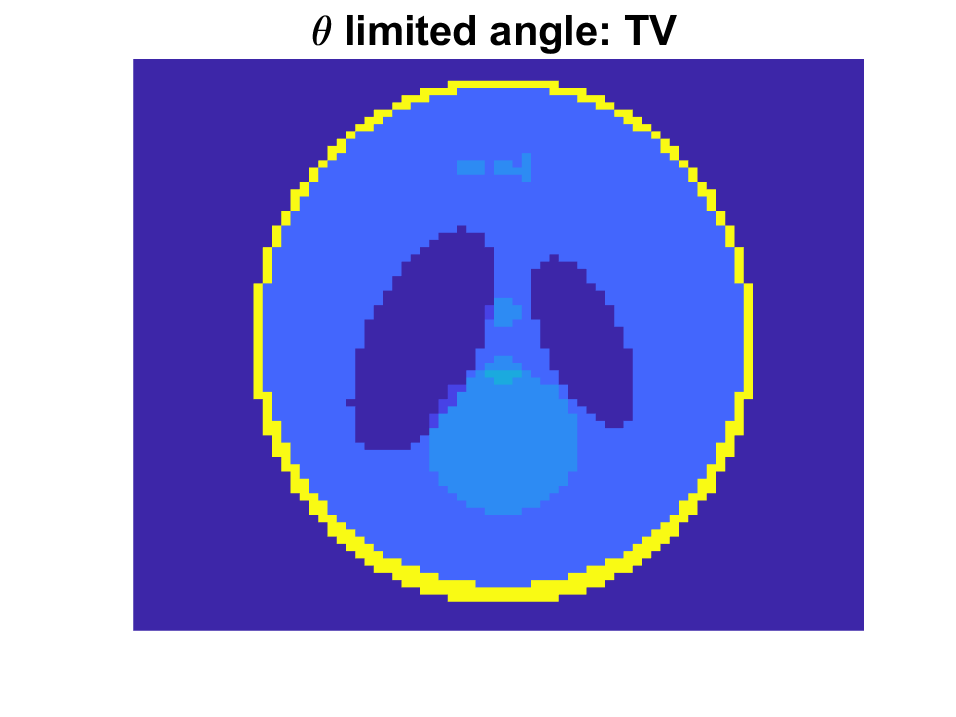} 
\caption{Comparision of the RTV filter with a variational TV regularization using a 
Pock-Chambolle method. Top row: RTV filter. Bottom row: TV-regularization. 
Left: full data. Center: Undersampled data with respect to $\sigma$.
Right: Limited angle data: $\theta \in [0,90^\circ]$.} \label{figcomp}
\end{figure}

Finally, we compare the RTV filter with a full variational TV regularization, i.e., a 
Tikhonov regularization with the \TV-seminorm as a regularizer. The computation of a minimizer is done by a Pock-Chambolle primal-dual iteration \cite{ChPo}, and the results are illustrated in Figure~\ref{figcomp} for the Shepp-Logan phantom. (Top: RTV. Bottom: TV-regularization). We used a rather coarse grid of $80 \times 80$ pixels for the phantom, since otherwise the computational complexity for the TV-regularization (both time and memory) becomes unreasonable.

The comparison is done for three cases of data: The results in the left column are obtained with full data, in the center column with undersampled data with respect to the offset variable $\sigma$ (only 50\% of lines are used), and in the right column with limited angle data with angles taken only from $[0,90^\circ]$. The case of RTV with undersampled data is computed as explained above, i.e., using a projection matrix $W$. 

The results clearly show the improvement gained from the TV-regularization; this is especially striking in the case of limited angle data. Note that the latter represents an unfavorable case for the RTV filter, since it does not involve $\theta$-derivatives, which limits its extrapolation capabilities in angular direction. In the case of $\sigma$-undersampling (center), also RTV performs better, while for full data the RTV method gives acceptable results. 

However, the improvement using TV-regularization comes with a drastically higher computational complexity; for instance, the computation for the RTV filter required about 0.01 seconds, while the 10 iterations for the Pock-Chambolle method required around 30 seconds (3000 times more); also memory requirements were an issue for TV regularization. This shows that the latter can hardly be used  with  practical devices, which is not a limitation for the RTV method.

\section{Conclusion}

We investigated norms for sinograms and conditions under which they are equivalent to some standard norms in image space. The Sharafutdinov norm was estimated against Bessel-potential norms in Theorem~\ref{mainshaf}. The $\no$-norm was compared to the usual total variation norm and equivalence was found for indicator functions of convex sets, cf.~Theorem~\ref{main1}, Corollary~\ref{maincor}, and for functions satisfying some symmetry (Theorem~\ref{th4}). Stability estimates for smooth functions were derived from Gagliardo-Nirenberg estimates. The results were used to construct nonlinear backprojection filters based on proximal mappings, which we then tested numerically on both simulated and experimental data.

\section{Acknowledgements and Support}

This research was funded in part by the Austrian Science Fund (FWF) SFB 10.55776/F68 ``Tomography Across the Scales'', project F6805-N36 (Tomography in Astronomy). For open access purposes, the authors have applied a CC BY public copyright license to any author-accepted manuscript version arising from this submission. The experimental THz-tomography data were recorded by Dr.~Peter Fosodeder during his employment at Recendt GmbH in the course of the ATTRACT project funded by the EC under Grant Agreement 777222. We want to thank a reviewer of a previous version of the manuscript and the reviewers of the current version for providing highly qualified comments which helped us to improve the article. 

{\footnotesize
\bibliographystyle{siam}
\bibliography{bib}
}

\end{document}